\documentclass[12pt,twoside]{article}
\usepackage[left=1.5cm,right=1.5cm,top=3cm,bottom=2cm]{geometry}
\usepackage{paralist}
\usepackage{marginnote}
\usepackage{amsfonts}
\usepackage{comment}
\usepackage{enumitem}
\usepackage{tikz}
\usetikzlibrary{matrix}
\usepackage{amsthm}
\usepackage{amsrefs}
\usepackage{amssymb}
\usepackage{amsmath}
\usepackage{graphicx}
\usepackage{varwidth}
\usepackage{hyperref}
\usepackage{fancyhdr}
\usepackage{stmaryrd}
\usepackage{multirow}

\usepackage{blkarray}
\usepackage[autostyle]{csquotes}
\usepackage[mathscr]{euscript}
\hypersetup{
    colorlinks=true,
    linkcolor=black,
    filecolor=magenta,      
    urlcolor=black,
    citecolor=black
}
\DeclareMathAlphabet{\mathcalligra}{T1}{calligra}{c}{h}
\providecommand{\U}[1]{\protect\rule{.1in}{.1in}}
\newtheorem{theorem}{Theorem}[section]
\newtheorem{proposition}[theorem]{Proposition}
\newtheorem{lemma}[theorem]{Lemma}

\newtheorem{corollary}[theorem]{Corollary}

\newtheorem{remark}[theorem]{Remark}
\let\oldremark\remark
\renewcommand{\remark}{\oldremark\normalfont}

\newtheorem{example}[theorem]{Example}
\let\oldexample\example
\renewcommand{\example}{\oldexample\normalfont}

\def\<{{\langle}}
\def\>{{\rangle}}

\def\bea{\begin{eqnarray*} }
\def\eea{\end{eqnarray*} }
\def\be{\begin{equation} }
\def\ee{\end{equation} }

\def\qed{\ifhmode\unskip\nobreak\fi\ifmmode\ifinner
\else\hskip5 pt \fi\fi\hbox{\hskip5 pt \vrule width4 pt
height6 pt  depth1.5 pt \hskip 1pt }}

\DeclareMathOperator*{\Vol}{Vol}

\DeclareMathOperator*{\supp}{supp}

\DeclareMathOperator*{\expo}{exp}

\DeclareMathOperator*{\grad}{grad}
\DeclareMathOperator*{\ess}{ess}
\DeclareMathOperator*{\Lip}{Lip}

\begin{document}

\title{Coverings preserving the bottom of the spectrum}
\author{Panagiotis Polymerakis}
\date{}

\maketitle

\renewcommand{\thefootnote}{\fnsymbol{footnote}}
\footnotetext{\emph{Date:} \today} 
\renewcommand{\thefootnote}{\arabic{footnote}}

\renewcommand{\thefootnote}{\fnsymbol{footnote}}
\renewcommand{\thefootnote}{\arabic{footnote}}

\renewcommand{\thefootnote}{\fnsymbol{footnote}}
\renewcommand{\thefootnote}{\arabic{footnote}}

\begin{abstract}
We prove that if a Riemannian covering preserves the bottom of the spectrum of a Schr\"{o}dinger operator, which belongs to the discrete spectrum of the operator on the base manifold, then the covering is amenable.


\end{abstract}

\section{Introduction}

The spectrum of the Laplacian on a Riemannian manifold is a natural isometric invariant. However, its behavior under maps between Riemannian manifolds, which respect the geometry of the manifolds to some extent, remains largely unclear. In this paper, we study the behavior of the bottom (that is, the minimum) of the spectrum under Riemannian coverings.

Let $p \colon M_{2} \to M_{1}$ be a Riemannian covering, $S_{1} = \Delta + V$ a Schr\"{o}dinger operator on $M_{1}$, with $V$ smooth and bounded from below, and $S_{2} = \Delta + V \circ p$ its lift on $M_{2}$. Then the bottoms of their spectra always satisfy the inequality $\lambda_{0}(S_{1}) \leq \lambda_{0}(S_{2})$. It is natural to examine when the equality holds.
Brooks \cite{Brooks1} proved that if the underlying manifold is closed (that is, compact without boundary), then a normal covering $p$ preserves the bottom of the spectrum of the Laplacian if and only if $p$ is amenable.

This theorem motivated the study of the behavior of the bottom of the spectrum under amenable coverings. Extending some former results \cite{Brooks2,MR3104995}, it was proved in \cite{BMP1} that amenable Riemannian coverings preserve the bottom of the spectrum of Schr\"{o}dinger operators, without any topological or geometric assumptions on the manifolds. In \cite{Mine}, it was proved that if, in addition, $M_{1}$ is complete, then the spectra of the operators satisfy $\sigma(S_{1}) \subset \sigma(S_{2})$. If, in addition, the covering is infinite sheeted, then $\sigma(S_{1}) \subset \sigma_{\ess}(S_{2})$, where $\sigma_{\ess}$ stands for the essential spectrum of the operator.

Although amenability of the covering is a natural assumption for the preservation of the bottom of the spectrum, it is not clear to what extent it is optimal. In this direction, Brooks \cite{Brooks2}, and Roblin and Tapie \cite{RT}, proved that under some quite restrictive assumptions, if the bottom of the spectrum of the Laplacian is preserved, then the covering is amenable. These assumptions involve the spectrum of fundamental domains of the covering and in particular, imply that the bottom of the spectrum of the Laplacian on $M_{1}$ belongs to its discrete spectrum (that is, the bottom is an isolated point of the spectrum and in particular, an eigenvalue). Moreover, in both results, the covering is assumed to be normal, with finitely generated deck transformations group.
Recently, in \cite{BMP2}, these conditions were replaced with some more natural geometric assumptions. 
More precisely, it was proved that if the manifolds are complete, without boundary, with Ricci curvature bounded from below, $V$ and $\grad V$ are bounded, the bottom of the spectrum is preserved, and belongs to the discrete spectrum of $S_{1}$, then the covering is amenable. A question raised in \cite{BMP2} is whether the assumption on the Ricci curvature is necessary. In this paper, we deal with this question and establish a generalization of all the above results.

Initially, using the result of \cite{BMP2} we prove the following analogue of Brooks' theorem \cite{Brooks1}, involving the bottom of the Neumann spectrum of the Laplacian on manifolds with (smooth) boundary.

\begin{theorem}\label{Theorem Neumann}
Let $p \colon M_{2} \to M_{1}$ be a Riemannian covering, with $M_{1}$ compact with boundary. Then $\lambda_{0}^{N}(M_{2}) = 0$ if and only if $p$ is amenable. 
\end{theorem}

The fact that amenable coverings preserve the bottom of the Neumann spectrum was essentially established in \cite{Mine}. The main point in the above theorem is the converse implication, which is the first result providing amenability of a covering of manifolds with boundary. This turns out to be quite useful in the study of arbitrary Riemannian coverings. More precisely, as an application of this theorem, we prove the following result.


\begin{theorem}\label{Main result intro}
Let $p \colon M_{2} \to M_{1}$ be a Riemannian covering. Let $S_{1}$ be a Schr\"{o}dinger operator on $M_{1}$, with $\lambda_{0}(S_{1}) \notin \sigma_{\ess}(S_{1})$, and $S_{2}$ its lift on $M_{2}$. Then $\lambda_{0}(S_{2}) = \lambda_{0}(S_{1})$ if and only if $p$ is amenable.
\end{theorem}

It is worth to point out that we do not impose any topological or geometric assumptions on the manifolds. Hence, Theorem \ref{Main result intro} is more general than the results of \cite{BMP2, Brooks2, RT}, since their assumptions imply that $\lambda_{0}(S_{1}) \notin \sigma_{\ess}(S_{1})$. Examining the optimality of the assumption $\lambda_{0}(S_{1}) \notin \sigma_{\ess}(S_{1})$ in this theorem, we show that it cannot be replaced with $\lambda_{0}(S_{1})$ being an eigenvalue.

Since the manifolds in the above theorem may be non-complete, we obtain immediately the corresponding result for Dirichlet spectra of Schr\"{o}dinger operators on manifolds with boundary. For sake of completeness, we also establish the corresponding result for the Neumann spectra, obtaining a generalization of Theorem \ref{Theorem Neumann}.
In particular, we obtain analogues of Brooks' result \cite{Brooks1} for Schr\"{o}dinger operators on manifolds with boundary. Namely, it follows that a Riemannian covering of a compact manifold is amenable if and only if it preserves the bottom of the Dirichlet/Neumann spectrum of some/any Schr\"{o}dinger operator. In virtue of \cite[Theorem 1]{Mine}, this is actually equivalent to the inclusion of the Dirichlet (Neumann) spectrum of a Schr\"{o}dinger operator on $M_{1}$, in the Dirichlet (Neumann, respectively) spectrum of its lift on $M_{2}$. The corresponding statement for manifolds without boundary has been established in \cite[Theorem 1.5]{Mine}.

Finally, as another application of Theorem \ref{Theorem Neumann}, we prove that if an infinite sheeted Riemannian covering preserves the bottom of the spectrum of a Schr\"{o}dinger operator, then the bottom of the spectrum belongs to the essential spectrum of the operator on the covering space. This was observed for the Laplacian in \cite{BMP2}. For sake of completeness, we establish the analogous result for the Dirichlet and Neumann spectra of Schr\"{o}dinger operators on manifolds with boundary.





The paper is organized as follows: In Section \ref{Section Preliminaries}, we give some preliminaries. In Section \ref{Section Spectrum}, we present some properties of the spectrum of Schr\"{o}dinger operators. In Section \ref{Section Compact}, we study Riemannian coverings of compact manifolds and establish Theorem \ref{Theorem Neumann}. In Section \ref{Section Arbitrary}, we study arbitrary Riemannian coverings and prove Theorem \ref{Main result intro} and the corresponding results for manifolds with boundary. In Section \ref{Section Application}, we show the aforementioned application of Theorem \ref{Theorem Neumann} for infinite sheeted coverings. \medskip

\textbf{Acknowledgements.} I would like to thank Werner Ballmann and Dorothee Sch\"{u}th for their helpful comments and remarks. I am also grateful to the Max Planck Institute for Mathematics in Bonn for its support and hospitality

\section{Preliminaries}\label{Section Preliminaries}

We begin by recalling some definitions and standard facts from functional analysis, which may be found for instance, in \cite{MR1335452} and \cite[Appendix A]{Taylor1}.

Let $L \colon \mathcal{D}(L) \subset \mathcal{H} \to \mathcal{H}$ be a self-adjoint operator on a separable Hilbert space $\mathcal{H}$, over $\mathbb{R}$ or $\mathbb{C}$. The \textit{spectrum} of $L$ is given by
\[
\sigma(L) := \{ \lambda \in \mathbb{R} : (L - \lambda) \colon \mathcal{D}(L) \subset \mathcal{H} \to \mathcal{H} \text{ not invertible} \}.
\]
The \textit{essential spectrum} of $L$ is defined as
\[
\sigma_{\ess}(L) := \{ \lambda \in \mathbb{R} : (L - \lambda) \colon \mathcal{D}(L) \subset \mathcal{H} \to \mathcal{H} \text{ not Fredholm} \}.
\]
Recall that an operator is called \textit{Fredholm} if its kernel is finite dimensional and its range is closed and of finite codimension. The \textit{discrete spectrum} of $L$ is given by $\sigma_{d}(L) := \sigma(L) \smallsetminus \sigma_{\ess}(L)$, and consists of isolated eigenvalues of $L$ of finite multiplicity.

The spectrum of a self-adjoint operator is a closed subset of $\mathbb{R}$. If $\sigma(L)$ is bounded from below, then its minimum is called the \textit{bottom} of the spectrum of $L$ and is denoted by $\lambda_{0}(L)$. 
The following characterization is due to Rayleigh.
\begin{proposition}\label{Rayleigh}
If $\sigma(L)$ is bounded from below, then the bottom of the spectrum of $L$ is given by
\[
\lambda_{0}(L) = \inf_{v \in \mathcal{D}(L) \smallsetminus \{0\}} \frac{\langle L v, v \rangle_{\mathcal{H}}}{\| v \|_{\mathcal{H}}^{2}}.
\]
\end{proposition}

Let $T \colon \mathcal{D}(T) \subset \mathcal{H} \to \mathcal{H}$ be a densely defined, symmetric linear operator. Assume that $T$ is \textit{bounded from below}, that is, there exists $c \in \mathbb{R}$, such that
\begin{equation}\label{bounded from below}
\langle T v,v \rangle_{\mathcal{H}} \geq c \| v \|_{\mathcal{H}}^{2},
\end{equation}
for all $v \in \mathcal{D}(T)$. Fix such a $c \in \mathbb{R}$ (not necessarily the supremum of all $c$ for which (\ref{bounded from below}) holds) and consider the inner product
\[
\langle v_{1}, v_{2} \rangle_{\mathcal{H}_{1}} := \langle Tv_{1},v_{2}\rangle_{\mathcal{H}} + (1-c) \langle v_{1} , v_{2} \rangle_{\mathcal{H}}
\]
on $\mathcal{D}(T)$. Let $\mathcal{H}_{1}$ be the completion of $\mathcal{D}(T)$ with respect this inner product. Then $\mathcal{H}_{1}$ can be identified with a dense subspace of $\mathcal{H}$, via a continuous injection. The domain of the Friedrichs extension $T^{(F)}$ of $T$ is given by
\[
\mathcal{D}(T^{(F)}) := \{ v \in \mathcal{H}_{1} : \text{there exists } v^{\prime} \in \mathcal{H}, \text{ such that } \langle v^{\prime} , w \rangle_{\mathcal{H}} = \langle v,w \rangle_{\mathcal{H}_{1}} \text{, for all } w \in \mathcal{H}_{1} \}.
\]
For $v \in \mathcal{D}(T^{(F)})$, we define $T^{(F)}v := v^{\prime} + (c-1) v$. Then $T^{(F)}$ is a self-adjoint extension of $T$ and is called the \textit{Friedrichs extension} of $T$.

\begin{proposition}\label{Bottom of Friedrichs}
The bottom of the spectrum of the Friedrichs extension of $T$ is given by
\[
\lambda_{0}(T^{(F)}) = c-1 +  \inf_{v \in \mathcal{H}^{\prime} \smallsetminus\{0\}} \frac{\| v \|_{\mathcal{H}_{1}}^{2}}{\| v \|_{\mathcal{H}}^{2}},
\]
where the infimum may be taken over any subspace $\mathcal{H}^{\prime}$, with $\mathcal{D}(T) \subset \mathcal{H}^{\prime} \subset \mathcal{H}_{1}$.
\end{proposition}

\begin{proof}
Evidently, for a non-zero $v \in \mathcal{D}(T^{(F)})$, we have
\[
c-1 + \frac{\| v \|_{\mathcal{H}_{1}}^{2}}{\| v \|_{\mathcal{H}}^{2}} = \frac{\langle T^{(F)}v, v \rangle_{\mathcal{H}}}{\| v \|_{\mathcal{H}}^{2}}.
\]
From Proposition \ref{Rayleigh}, we obtain the asserted equality, where the infimum is taken over all $v \in \mathcal{D}(T^{(F)}) \smallsetminus \{0\}$. From the definition of $\mathcal{H}_{1}$, it is easy to see that we obtain the same infimum for $v \in \mathcal{D}(T) \smallsetminus \{0\}$ and for $v \in \mathcal{H}_{1} \smallsetminus \{0\}$. \qed
\end{proof}
\subsection{Schr\"{o}dinger operators}\label{Schroedinger preliminaries}

Throughout this paper, manifolds are assumed to be connected with not necessarily connected, possibly empty, smooth boundary, unless otherwise stated. Let $M$ be a possibly non-connected Riemannian manifold. A \textit{Schr\"{o}dinger operator} on $M$ is an operator of the form $S = \Delta + V$, where $\Delta$ is the (non-negative definite) Laplacian and $V \colon M \to \mathbb{R}$ is smooth and bounded from below.
On the space $C^{\infty}_{c}(M)$ consider the inner product
\[
\langle f , g \rangle_{H_{V}(M)} := \int_{M} (\langle \grad f, \grad g \rangle + (V - {\inf}_{M}V + 1)f g).
\]


If $M$ has empty boundary, let $H_{V}(M)$ be the completion of $C^{\infty}_{c}(M)$ with respect to this inner product. If $M$ has non-empty boundary, let $H_{V}(M)$ be the completion of $\{ f \in C^{\infty}_{c}(M) : \nu(f) = 0 \text{ on } \partial M \}$ with respect to this inner product, where $\nu$ is the inward pointing normal to $\partial M$. It is clear that $H_{V}(M)$ can be identified with a dense subspace of $L^{2}(M)$, via a continuous injection.

If $M$ has empty boundary, we are interested in the Friedrichs extension of the operator
\[
S \colon C^{\infty}_{c}(M) \subset L^{2}(M) \to L^{2}(M).
\]
If $M$ has non-empty boundary, we are interested in the Neumann extension of $S$, that is, the Friedrichs extension of
\[
S \colon \{ f \in C^{\infty}_{c}(M) : \nu(f) = 0 \text{ on } \partial M \} \subset L^{2}(M) \to L^{2}(M).
\]
In any of these cases, we denote this Friedrichs extension by $S^{N}$ and its domain by $\mathcal{D}(S^{N})$. It is worth to point out that the space $H_{V}(M)$ plays the role of $\mathcal{H}_{1}$ in the discussion of the Friedrichs extension in the beginning of this section (where we consider the lower bound $c := \inf_{M}V$ for the operator).

The spectrum and the essential spectrum of $S^{N}$ are denoted by $\sigma^{N}(S)$ and $\sigma_{\ess}^{N}(S)$, respectively, and their bottoms (that is, their minimums) by $\lambda_{0}^{N}(S)$ and $\lambda_{0}^{N, \ess}(S)$, respectively. These sets and quantities for the Laplacian are denoted by $\sigma^{N}(M)$, $\sigma_{\ess}^{N}(M)$ and $\lambda_{0}^{N}(M)$, $\lambda_{0}^{N, \ess}(M)$, respectively.
If $M$ has empty boundary, we sometimes drop the superscript ``$N$" in the notation of the spectrum, the essential spectrum and their bottoms.

If $M$ has non-empty boundary, the Dirichlet extension $S^{D}$ of $S$ is the Friedrichs extension of the operator
\[
S \colon \{ f \in C^{\infty}_{c}(M) : f = 0 \text{ on } \partial M \} \subset L^{2}(M) \to L^{2}(M).
\]
We denote by $\lambda_{0}^{D}(S)$ the bottom of the spectrum of this operator. The bottom of the spectrum of the Dirichlet extension of the Laplacian is denoted by $\lambda_{0}^{D}(M)$.
In virtue of the next remark, Dirichlet extensions of Schr\"{o}dinger operators are closely related to Schr\"{o}dinger operators on non-complete manifolds without boundary.

\begin{remark}\label{Dirichlet}
If $M$ is a Riemannian manifold with non-empty boundary, then any $f \in C^{\infty}_{c}(M)$ vanishing on $\partial M$, can be approximated in $H^{1}(M)$ with smooth functions, compactly supported in the interior of $M$.
Therefore, if $S$ is a Schr\"{o}dinger operator on $M$, then the Dirichlet extension of $S$ coincides with the Friedrichs extension of $S$ viewed as an operator in the interior of $M$.
\end{remark}


We end this subsection with some already known properties of the spectrum, that will be used in the sequel. Since they will be used only for complete manifolds without boundary, we do not state them in their most general forms.

The next proposition characterizes the bottom of the spectrum of a Schr\"{o}dinger operator as the maximum of its positive spectrum, and may be found in \cite[Theorem 7]{Cheng-Yau}, \cite[Theorem 1]{MR562550} and \cite[Theorem 2.1]{MR882827}.

\begin{proposition}\label{Maximum of positive spectrum}
Let $S$ be a Schr\"{o}dinger operator on a complete Riemannian manifold $M$ without boundary. Then $\lambda_{0}(S)$ is the maximum of all $\lambda \in \mathbb{R}$, such that there exists a positive $\varphi \in C^{\infty}(M)$, with $S \varphi = \lambda \varphi$.
\end{proposition}

It is worth to point out that the positive functions involved in this proposition are not required to be square-integrable.
The next expression of the bottom of the essential spectrum follows from the Decomposition Principle \cite[Proposition 2.1]{Donnelly-Li}.

\begin{proposition}[{\cite[Proposition 3.2]{MR2891739}}]\label{Bottom of essential spectrum}
Let $S$ be a Schr\"{o}dinger operator on a complete Riemannian manifold $M$ without boundary. Let $(K_{n})_{n \in \mathbb
N}$ be an exhausting sequence of $M$ consisting of compact sets. Then the bottom of the essential spectrum of $S$ is given by
\[
\lambda_{0}^{\ess}(S) = \lim_{n} \lambda_{0}(S,M \smallsetminus K_{n}),
\]
where $\lambda_{0}(S,M \smallsetminus K_{n})$ stands for the bottom of the spectrum of $S$ on $M \smallsetminus K_{n}$.
\end{proposition}

\subsection{Amenable coverings}

In this subsection, we present the definition and some basic properties of amenable coverings. A right action of a countable group $\Gamma$ on a countable set $X$ is called \textit{amenable} if there exists a $\Gamma$-invariant mean on $L^{\infty}(X)$.

Throughout the paper manifolds are assumed to be connected, unless otherwise stated. In particular, Riemannian coverings are assumed to be between connected manifolds, unless otherwise stated. For reasons that will become clear in the sequel, we must consider possibly non-connected covering spaces at some points.

Let $p \colon M_{2} \to M_{1}$ be a Riemannian covering, with $M_{2}$ possibly non-connected, that is, $M_{2}$ has countably many connected components, the restriction of $p$ on any component is a Riemannian covering over $M_{1}$, and any point of $M_{1}$ has a neighborhood that is evenly covered with respect to the restriction of $p$ on any connected component of $M_{2}$. Fix $x \in M_{1}^{\circ}$ (that is, the interior of $M_{1}$) and consider the fundamental group $\pi_{1}(M_{1})$ of $M_{1}$ with base point $x$. For $g \in \pi_{1}(M_{1})$, let $\gamma_{g} \colon [0,1] \to M_{1}$ be a representative loop based at $x$. For $y \in p^{-1}(x)$, lift $\gamma_{g}$ to a path $\tilde{\gamma}_{g}$, with $\tilde{\gamma}_{g}(0) = y$. We define $y \cdot g := \tilde{\gamma}_{g}(1)$. In this way, we obtain a right action of $\pi_{1}(M_{1})$ on $p^{-1}(x)$. The covering $p$ is called \textit{amenable} if this right action is amenable.

This definition coincides with the definition presented in \cite{BMP2,BMP1,Mine} in terms of the right cosets of $\pi_{1}(M_{2})$ in $\pi_{1}(M_{1})$, when $M_{2}$ is connected. However, this definition allows us to extend the notion of amenable coverings in case $M_{2}$ is non-connected. 

For instance, consider a Riemannian covering $p \colon  M_{2} \to M_{1}$, where $M_{2}$ has countably many connected components $M_{2}^{(n)}$, $n \in \mathbb{N}$. If the restriction $p \colon M_{2}^{(n)} \to M_{1}$ is amenable, for some $n \in \mathbb{N}$, then the covering $p \colon M_{2} \to M_{1}$ is amenable. Indeed, if there exists a $\pi_{1}(M_{1})$-invariant mean $\mu_{n}$ on $L^{\infty}(p^{-1}(x) \cap M_{2}^{(n)})$, for some $n \in \mathbb{N}$, then the linear functional $\mu \colon L^{\infty}(p^{-1}(x)) \to \mathbb{R}$, defined by
\[
\mu (f) := \mu_{n} (f|_{p^{-1}(x) \cap M_{2}^{(n)}}),
\]
for any $f \in L^{\infty}(p^{-1}(x))$, is a $\pi_{1}(M_{1})$-invariant mean on $L^{\infty}(p^{-1}(x))$.
However, the covering $p \colon M_{2} \to M_{1}$ may be amenable, even when the restriction $p \colon M_{2}^{(n)} \to M_{1}$ is non-amenable, for any $n \in \mathbb{N}$.


The following characterization of amenable coverings follows from F\o{}lner's criterion (cf. \cite[Section 2]{MR3104995}).

\begin{proposition}\label{Folner}
The covering $p$ is amenable if and only if for any finite $G \subset \pi_{1}(M_{1})$ and $\varepsilon > 0$, there exists a finite $F \subset p^{-1}(x)$, such that
\[
\#(F \smallsetminus Fg) < \varepsilon \#(F),
\]
for all $g \in G$.
\end{proposition}

In particular, the covering is amenable if and only if the right action of any finitely generated subgroup of $\pi_{1}(M_{1})$ on $p^{-1}(x)$ is amenable. For a smoothly bounded, compact and connected neighborhood $K$ of $x$, we denote by $i_{*}\pi_{1}(K)$ the image of the fundamental group of $K$ in $\pi_{1}(M_{1})$.
It is clear that $p \colon  p^{-1}(K) \to K$ is a Riemannian covering of manifolds with boundary, where $p^{-1}(K)$ is possibly non-connected.
Evidently, the covering $p \colon p^{-1}(K) \to K$ is amenable if and only if the right action of $i_{*}\pi_{1}(K)$ on $p^{-1}(x)$ is amenable.

\begin{proposition}\label{amenability domains}
The covering $p \colon M_{2} \to M_{1}$ is amenable if and only if the covering $p \colon p^{-1}(K) \to K$ is amenable, for any smoothly bounded, compact and connected neighborhood $K$ of $x$.
\end{proposition}

\begin{proof}
From Proposition \ref{Folner}, it suffices to prove that for any finite subset $G$ of $\pi_{1}(M_{1})$, there exists a smoothly bounded, compact and connected neighborhood $K$ of $x$, such that $G \subset i_{*}\pi_{1}(K)$. Let $G$ be a finite subset of $\pi_{1}(M_{1})$ and consider a representative loop $\gamma_{g} \colon [0,1] \to M_{1}^{\circ}$, for each $g \in G$. Let $C$ be the union of the images of these loops and let $U$ be a relatively compact, open neighborhood of $C$ that does not intersect the boundary of $M$ (if non-empty). Consider $\chi \in C^{\infty}(M_{1})$, with $0 \leq \chi \leq 1$, $\chi = 1$ in $C$ and $\supp \chi \subset U$. From Sard's Theorem, it follows that for almost any $t \in (0,1)$, the level set $\{ \chi = t \}$ is a smooth hypersurface of $M_{1}$. Consider such a $t$, and the smoothly bounded, compact set $K^{\prime} := \{ \chi \geq t \}$. Then for the connected component $K$ of $K^{\prime}$ containing $x$, we have $G \subset i_{*}\pi_{1}(K)$. \qed
\end{proof}

\subsection{Manifolds with Ricci curvature bounded from below}

In this subsection we recall the main result of \cite{BMP2} and point out that its proof, with some slight modifications, establishes this result for possibly non-connected covering spaces.

A non-connected Riemannian manifold $M$ is complete if all of its connected components are complete. The distance of points of different connected components of $M$ is considered to be infinite. In particular, any bounded subset of $M$ is contained in a connected component of $M$.


\begin{theorem}[{{\cite[Theorem 4.1]{BMP2}}}]\label{weak converse}
Let $p \colon M_{2} \to M_{1}$ be a Riemannian covering, with $M_{2}$ possibly non-connected. Assume that $M_{1}$ is complete, without boundary, and with Ricci curvature bounded from below. Let $S_{1} := \Delta + V$ be a Schr\"{o}dinger operator on $M_{1}$, with $V$ and $ \grad V $ bounded, and let $S_{2}$ be its lift on $M_{2}$. If $\lambda_{0}(S_{2}) = \lambda_{0}(S_{1}) \neq \lambda_{0}^{\ess}(S_{1})$, then the covering is amenable.
\end{theorem}

We begin with some definitions and remarks from \cite{BMP2}.
Let $M$ be a possibly non-connected Riemannian manifold without boundary. A positive $\varphi \in C^{\infty}(M)$ \textit{satisfies a Harnack estimate} if there exists a constant $c_{\varphi} \geq 1$, such that
\[
\sup_{B(x,r)}\varphi^{2} \leq c_{\varphi}^{r+1} \inf_{B(x,r)}\varphi^{2},
\]
for all $x \in M$ and $r>0$. 
Assume that $M$ is complete, with Ricci curvature bounded from below, and let $S = \Delta + V$ be a Schr\"{o}dinger operator on $M$, with $V$ and $ \grad V$ bounded. From \cite[Theorem 6]{Cheng-Yau}, if a positive function $\varphi \in C^{\infty}(M)$ satisfies $S\varphi = \lambda \varphi$, for some $\lambda \in \mathbb{R}$, then $\varphi$ satisfies a Harnack estimate.


The \textit{modified Cheeger's constant} of $M$ is defined as
\[
h_{\varphi}(M) := \inf_{A} \frac{\int_{\partial A} \varphi^{2}}{\int_{A} \varphi^{2}},
\]
where the infimum is taken over all bounded domains $A$ of $M$ with smooth boundary.

\begin{lemma}\label{aux1}
Let $M$ be a possibly non-connected, complete Riemannian manifold, without boundary and with Ricci curvature bounded from below. Let $\varphi \in C^{\infty}(M)$ be a positive function, which satisfies a Harnack estimate. If $h_{\varphi}(M) = 0$, then for any $\varepsilon, r  > 0$, there exists a bounded open subset $A$ of $M$, such that
\[
\int_{A^{r} \smallsetminus A} \varphi^{2} < \varepsilon \int_{A} \varphi^{2},
\]
where $A^{r}:=\{ y \in M : d(y,A) < r \}$.
\end{lemma}

\begin{proof}
We may renormalize the Riemannian metric of $M$, so that $\text{Ric}_{M} \geq 1 -m$, where $m$ is the dimension of $M$. Since $h_{\varphi}(M) = 0$, for any $\varepsilon, r > 0$, there exists a non-empty, bounded domain $A$ of $M$ satisfying the estimate (3.2) of \cite{BMP2}. Evidently, $A$ is contained in a connected component of $M$ and the arguments of the proof of \cite[Lemma 3.1]{BMP2} can be carried out in this connected component of $M$, establishing the asserted claim. \qed
\end{proof}


\begin{lemma}\label{aux2}
In the setting of Theorem \ref{weak converse}, there exists a compact set $K \subset M_{1}$, such that for any $\varepsilon,r>0$, there exists $z \in K$ and a bounded open subset $A$ of $M_{2}$, such that
\[
\#(p^{-1}(z) \cap (A^{r} \smallsetminus A)) < \varepsilon \#(p^{-1}(z) \cap A).
\]
\end{lemma}

\begin{proof}
Since $\lambda_{0}(S_{1}) \notin \sigma_{\ess}(S_{1})$, from Proposition \ref{Bottom of essential spectrum}, there exists a compact $K \subset M_{1}$, such that $\lambda_{0}(S_{1} , M_{1} \smallsetminus K) > \lambda_{0}(S_{1})$. The proof is identical to the one of \cite[Lemma 4.5]{BMP2}, taking into account that \cite[Lemma 3.1]{BMP2} has been extended to possibly non-connected manifolds in Lemma \ref{aux1}.\qed
\end{proof}\medskip


\noindent\emph{Proof of Theorem \ref{weak converse}:}
Fix $x \in M_{1}$ and consider the fundamental group $\pi_{1}(M_{1})$ with base point $x$. Consider a compact set $K \subset M_{1}$ as in Lemma \ref{aux2}, and let $R>0$, such that $K \subset B(x,R)$.
Let $\varepsilon > 0$ and $G$ be a finite subset of $\pi_{1}(M_{1})$. For each $g \in G$, consider a smooth representative loop $\gamma_{g}$ based at $x$, and let
\[
r > \max_{g \in G} \ell(\gamma_{g}) + 2R,
\]
where $\ell(\cdot)$ stands for the length of a curve.
From Lemma \ref{aux2}, there exists $z \in K$ and a bounded open subset $A$ of $M_{2}$, such that
\[
\#(p^{-1}(z) \cap (A^{r} \smallsetminus A)) < \varepsilon \#(p^{-1}(z) \cap A).
\]

Consider a smooth path $\gamma \colon [0,1] \to M_{1}$ from $x$ to $z$, of length less than $R$. For $y \in p^{-1}(x)$, lift $\gamma$ to a path $\tilde{\gamma} \colon [0,1] \to M_{2}$, with $\tilde{\gamma}(0) = y$, and define $\Phi(y) := \tilde{\gamma}(1)$. Then the map $\Phi \colon p^{-1}(x) \to p^{-1}(z)$ is bijective.
Let $F := \Phi^{-1}(p^{-1}(z) \cap A)$ and consider $y \in F \smallsetminus Fg$, for some $g \in G$. Then $\Phi(y) \in A$ and $\Phi(y\cdot g^{-1}) \notin A$. Evidently, we have
\[
d(\Phi(y) , \Phi(y \cdot g^{-1}) ) \leq d(y,y \cdot g^{-1}) + 2 \ell (\gamma) \leq \ell(\gamma_{g}) + 2 R < r.
\]
Therefore, $\Phi(y \cdot g^{-1}) \in A^{r} \smallsetminus A$. Since $\Phi$ is bijective, it is clear that
\begin{eqnarray}
\#(F \smallsetminus Fg) &=& \#\{y \cdot g^{-1} : y \in F \smallsetminus Fg\} = \#\{ \Phi (y \cdot g^{-1}) : y \in F \smallsetminus Fg\} \nonumber \\
&\leq& \#(p^{-1}(z) \cap (A^{r} \smallsetminus A)) < \varepsilon \#(p^{-1}(z) \cap A) = \varepsilon \#(F). \nonumber
\end{eqnarray}
From Proposition \ref{Folner}, it follows that the covering is amenable. \qed

\section{Properties of the Neumann spectrum}\label{Section Spectrum}

In this section, we establish some properties of the Neumann spectrum that will be used in the sequel. Let $M$ be a possibly non-connected Riemannian manifold and $S = \Delta + V$ a Schr\"{o}dinger operator on $M$. It is worth to point out that we do not require $M$ to have non-empty boundary, which yields that the following results also hold for manifolds without boundary (and most of them are already known in this case). If $M$ has non-empty boundary, we denote by $\nu$ the inward pointing normal to $\partial M$. Throughout this section, we denote by $H_{V}(M)$ the space defined in Subsection \ref{Schroedinger preliminaries}

First, we establish some convenient expressions for the bottom of the Neumann spectrum, and derive some straightforward applications to Riemannian coverings.

\begin{proposition}\label{Approximation}
Any compactly supported smooth function belongs to $H_{V}(M)$. Moreover, any compactly supported Lipschitz function is in $H_{V}(M)$.
\end{proposition}

\begin{proof}
If $M$ has empty boundary, then any compactly supported Lipschitz function $f$ belongs to $H_{0}^{1}(M)$. Since $V$ is smooth, it is easy to see that any such $f$ also belongs to $H_{V}(M)$. Therefore, it remains to prove the proposition for manifolds with non-empty boundary. 

Let $f \in C^{\infty}_{c}(M)$. Then there exists a compact $K \subset \partial M$ and $\delta > 0$, such that the map $\Phi \colon K \times [0,\delta) \to M$, defined by $\Phi(x,t) := \expo_{x} (t\nu)$, is a diffeomorphism onto its image $W_{\delta}$, and $\supp f \cap W_{\delta} \subset W_{\delta}^{\circ}$.
For $0 < \delta_{0} < \delta$, consider the Lipschitz function $f_{\delta_{0}}$, which is equal to $f$ outside $W_{\delta_{0}}$, and $f_{\delta_{0}}(\Phi(x,t)) = f(\Phi(x,\delta_{0}))$ in $W_{\delta_{0}}$.
Let $K_{1}$ be a compact neighborhood of $\Phi(K \times \{\delta_{0}\})$ and $K_{2}$ a compact neighborhood of $K_{1}$, that does not intersect $\partial M$. Consider $\chi \in C^{\infty}_{c}(M)$, with $\chi = 1$ in $K_{1}$ and $\supp \chi \subset K_{2}$. Since $\chi f_{\delta_{0}}$ is Lipschitz and compactly supported in the interior of $M$, it follows that $\chi f_{\delta_{0}} \in H_{V}(M)$.
Moreover, $(1 - \chi) f_{\delta_{0}} \in C^{\infty}_{c}(M)$ and $\nu(f) = 0$ on $\partial M$. Therefore, $(1-\chi)f_{\delta_{0}} \in H_{V}(M)$, which yields that $f_{\delta_{0}} \in H_{V}(M)$. It is clear that $f_{\delta_{0}} \rightarrow f$ in $H_{V}(M)$, as $\delta_{0} \rightarrow 0$, and in particular, $f \in H_{V}(M)$.

Let $f$ be a compactly supported Lipschitz function on $M$. Consider a Riemannian manifold $N$ of the same dimension, without boundary, containing $M$ (for instance, glue cylinders along $\partial M$). Extend $f$ to a compactly supported Lipschitz function $f^{\prime}$ in $N$ and let $K$ be a smoothly bounded, compact neighborhood of $\supp f^{\prime}$. Then there exists $(g_{n})_{n \in \mathbb{N}} \subset C^{\infty}_{c}(N)$, with $\supp g_{n} \subset K$ and $g_{n} \rightarrow f^{\prime}$ in $H_{0}^{1}(K)$. Then $h_{n} := g_{n}|_{M} \in C^{\infty}_{c}(M)$ and from the first statement, it follows that $h_{n} \in H_{V}(M)$. Evidently, we have that $h_{n} \rightarrow f$ in $H_{V}(M)$, and in particular, $f \in H_{V}(M)$. \qed
\end{proof}\medskip



For $f \in \Lip_{c}(M) \smallsetminus \{0\}$, the \textit{Rayleigh quotient} of $f$ with respect to $S$, is defined as
\[
\mathcal{R}_{S}(f) := \frac{\int_{M} (\| \grad f \|^{2} + Vf^{2})}{\int_{M} f^{2}}.
\]
\begin{proposition}\label{Bottom of Neumann}
The bottom of the spectrum of $S^{N}$ is given by
\[
\lambda_{0}^{N}(S) = \inf_{f \in C^{\infty}_{c}(M) \smallsetminus \{0\}} \mathcal{R}_{S}(f) = \inf_{f \in \Lip_{c}(M) \smallsetminus \{0\}} \mathcal{R}_{S}(f).
\]
\end{proposition}

\begin{proof}
It is clear that for any non-zero $f \in \Lip_{c}(M)$, we have
\[
\mathcal{R}_{S}(f) = {\inf}_{M}V - 1 + \frac{\| f \|_{H_{V}(M)}^{2}}{\| f \|_{L^{2}(M)}^{2}},
\]
and the asserted equalities follow from Proposition \ref{Bottom of Friedrichs}. \qed
\end{proof}

\begin{proposition}\label{Inequality of Bottoms}
Let $p \colon M_{2} \to M_{1}$ be a Riemannian covering, with $M_{2}$ possibly non-connected. Let $S_{1}$ be a Schr\"{o}dinger operator on $M_{1}$ and consider its lift $S_{2}$ on $M_{2}$. Then $\lambda_{0}^{N}(S_{1}) \leq \lambda_{0}^{N}(S_{2})$.
\end{proposition}

\begin{proof}
Let $f \in C^{\infty}_{c}(M_{2}) \smallsetminus \{0\}$ and consider its pushdown
\[
g(z) := \big( \sum_{y \in p^{-1}(z)} f(y)^{2} \big)^{1/2}
\]
on $M_{1}$. Then $g \in \Lip_{c}(M_{1})$, $\| g \|_{L^{2}(M_{1})} = \| f \|_{L^{2}(M_{2})}$ and $\mathcal{R}_{S_{1}}(g) \leq \mathcal{R}_{S_{2}}(f)$ (cf. \cite[Section 4]{BMP1}). The statement follows from Proposition \ref{Bottom of Neumann}. \qed
\end{proof}

\begin{theorem}\label{Old result for Neumann}
Let $p \colon M_{2} \to M_{1}$ be a Riemannian covering. Let $S_{1}$ be a Schr\"{o}dinger operator on $M_{1}$ and consider its lift $S_{2}$ on $M_{2}$. If $p$ is infinite sheeted and amenable, then $\lambda_{0}^{N}(S_{1}) = \lambda_{0}^{N,\ess}(S_{2})$.
\end{theorem}

\begin{proof}
Follows from \cite[Theorem 1.2]{Mine} and Corollary \ref{Inequality of Bottoms}. \qed
\end{proof}\medskip

Next, we study properties of eigenfunctions corresponding to the bottom of the spectrum and minimizing sequences for the Rayleigh quotient of Schr\"{o}dinger operators on connected Riemannian manifolds.

\begin{proposition}\label{minimizing sequence}
Let $S=\Delta + V$ be a Schr\"{o}dinger operator on a Riemannian manifold $M$, and consider $(f_{n})_{n \in \mathbb{N}} \subset \Lip_{c}(M)$, with $\| f_{n} \|_{L^{2}(M)} = 1$ and $\mathcal{R}_{S}(f_{n}) \rightarrow \lambda_{0}^{N}(S)$. If $\lambda_{0}^{N}(S) 
\notin \sigma_{\ess}^{N}(S)$, then there exists a subsequence $(f_{n_{k}})_{k \in \mathbb{N}}$, such that $f_{n_{k}} \rightarrow \varphi$ in $L^{2}(M)$, for some $\lambda_{0}^{N}(S)$-eigenfunction $\varphi$ of $S^{N}$.
\end{proposition}

\begin{proof}
From Proposition \ref{Approximation}, there exists $(f_{n}^{\prime})_{n \in \mathbb{N}} \subset C^{\infty}_{c}(M) \cap \mathcal{D}(S^{N})$, with $\| f_{n}^{\prime} \|_{L^{2}(M)} = 1$ and $\| f_{n} - f_{n}^{\prime} \|_{H_{V}(M)} \leq 1/n$, for any $n  \in \mathbb{N}$. Evidently, $\mathcal{R}_{S}(f_{n}^{\prime}) \rightarrow \lambda_{0}^{N}(S)$ and it suffices to prove the asserted statement for $(f^{\prime}_{n})_{n \in \mathbb{N}}$.

Since $\lambda_{0}^{N}(S)$ is not in the essential spectrum, it is an isolated eigenvalue of finite multiplicity. Let $E$ be the eigenspace corresponding to $\lambda_{0}^{N}(S)$, and $g_{n}$ be the projection of $f^{\prime}_{n}$ on $E$ with respect to the $L^{2}(M)$-inner product, $n \in \mathbb{N}$. Since $E$ is finite dimensional, after passing to a subsequence, we may assume that $g_{n} \rightarrow \varphi$ in $L^{2}(M)$, for some $\varphi \in E$. Consider $h_{n} := f^{\prime}_{n} - g_{n} \in \mathcal{D}(S^{N})$. 
Since $h_{n}$ is perpendicular to $E$, from the Spectral Theorem (cf. for instance \cite[Chapter 8]{Taylor2}), it follows that there exists $c_{0}>0$, such that
\begin{equation}\label{eq}
\| h_{n} \|_{H_{V}(M)}^{2} - (1 - {\inf}_{M}V) \| h_{n} \|_{L^{2}(M)}^{2} = \langle S^{N} h_{n} , h_{n} \rangle_{L^{2}(M)} \geq (\lambda_{0}^{N}(S) + c_{0}) \| h_{n} \|_{L^{2}(M)}^{2} ,
\end{equation}
for any $n \in \mathbb{N}$. It is clear that
\begin{eqnarray}
\langle h_{n} , g_{n} \rangle_{H_{V}(M)} &=& \langle h_{n} , S^{N}g_{n} \rangle_{L^{2}(M)} + (1-{\inf}_{M}V) \langle h_{n} ,g_{n} \rangle_{L^{2}(M)} \nonumber \\
&=& (\lambda_{0}^{N}(S) + 1 - {\inf}_{M}V) \langle h_{n} , g_{n} \rangle_{L^{2}(M)} = 0. \nonumber
\end{eqnarray}
Let $\varepsilon > 0$. Then, for $n$ sufficiently large, we have $\mathcal{R}_{S}(f_{n}) \leq \lambda_{0}^{N}(S) + \varepsilon$, and thus
\begin{eqnarray}
\| h_{n} \|_{H_{V}(M)}^{2} - (1 - {\inf}_{M}V) \| h_{n} \|_{L^{2}(M)}^{2} &=& (\| f_{n}^{\prime} \|_{H_{V}(M)}^{2} - (1-{\inf}_{M}V) \| f_{n}^{\prime} \|_{L^{2}(M)}^{2}) \nonumber\\
&-& (\| g_{n} \|_{H_{V}(M)}^{2} - (1-{\inf}_{M}V) \| g_{n} \|_{L^{2}(M)}^{2} ) \nonumber \\
&\leq& (\lambda_{0}^{N}(S) + \varepsilon) \| f^{\prime}_{n} \|_{L^{2}(M)}^{2} - \lambda_{0}^{N}(S) \| g_{n} \|_{L^{2}(M)}^{2} \nonumber \\
&=& \varepsilon + \lambda^{N}_{0}(S) \| h_{n} \|_{L^{2}(M)}^{2}. \nonumber
\end{eqnarray}
From (\ref{eq}), this yields that $h_{n} \rightarrow 0$ in $L^{2}(M)$. Therefore, $f_{n}^{\prime} \rightarrow \varphi$ in $L^{2}(M)$.\qed

\end{proof}

\begin{lemma}\label{Eigenfunction auxiliary}
Let $S$ be a Schr\"{o}dinger operator on a (connected) Riemannian manifold $M$ and let $\varphi \in C^{\infty}(M) \smallsetminus \{0\}$ be a non-negative function satisfying $S \varphi = \lambda \varphi$, for some $\lambda \in \mathbb{R}$. Then $\varphi$ is positive in the interior of $M$. If, in addition, $M$ has non-empty boundary, and $\nu(\varphi) = 0$ on $\partial M$, then $\varphi$ is positive on $\partial M$.
\end{lemma}

\begin{proof}
Assume that there exists a point $x$ in the interior of $M$, such that $\varphi(x)=0$. Let $\delta > 0$, such that $\expo_{x} \colon B(0,2\delta) \subset T_{x}M \to M$ is a diffeomorphism onto its image. Then $B(x,\delta)$ may be considered as a geodesic ball of radius $\delta$ in a complete Riemannian manifold without boundary. In $B(x,\delta)$, for any $\varepsilon > 0$, we have
\[
| \Delta (\varphi + \varepsilon)| \leq (\varphi + \varepsilon) \sup_{B(x,\delta)} |\lambda - V|
\]
and
\[
\| \grad \Delta(\varphi + \varepsilon) \| \leq \|\grad (\varphi + \varepsilon)\| \sup_{B(x,\delta)} |\lambda - V|  + (\varphi + \varepsilon) \sup_{B(x,\delta)} \| \grad V\|.
\]
From \cite[Theorem 6]{Cheng-Yau}, it follows that there exists $c> 0$, independent from $\varepsilon$, such that $$\sup_{B(x,\delta/2)} (\varphi + \varepsilon) \leq c \inf_{B(x,\delta/2)} (\varphi + \varepsilon).$$
Since $\varepsilon > 0$ is arbitrary, it follows that if $\varphi(x) = 0$, then $\varphi = 0$ in $B(x,\delta/2)$. In particular, the set $\{ x \in M^{\circ} : \varphi(x) = 0 \}$ is open and closed. Since $M$ is connected and $\varphi$ is not identically zero, it follows that $\varphi$ is positive in $M^{\circ}$.

Assume that $M$ has non-empty boundary and $\nu(\varphi) = 0$ on $\partial M$. Assume that there exists $x \in \partial M$, such that $\varphi(x) = 0$. Since $\nu(\varphi) = 0$ on $\partial M$ and $\varphi|_{\partial M}$ attains a minimum at $x$, it follows that $\grad \varphi (x) = 0$.
Consider a coordinate system $\Phi \colon U := B(0,r) \cap \mathbb{H}^{m} \to M$, with $\Phi(0) = x$, where $m$ is the dimension of $M$ and $\mathbb{H}^{m}$ is the upper half-space of dimension $m$. 
Consider $c_{0} \in \mathbb{R}$, with $c_{0} \geq - \inf_{M} V$ and $c_{0} \geq -\lambda$.
Then $\phi := \varphi \circ \Phi$ is non-negative, smooth and satisfies
\[
L \phi := - \frac{1}{\sqrt{ \det g }} \sum_{i,j=1}^{m}  \frac{\partial}{\partial x_{i}} (g^{ij} \sqrt{ \det g } \frac{\partial \phi}{\partial x_{j}}) + (V + c_{0}) \phi = (\lambda + c_{0}) \phi \geq 0.
\]
Since $V + c_{0} \geq 0$, $\phi(0) = 0 < \phi(y)$ for all $y \in U^{\circ}$, and $U$ satisfies the interior ball condition at the origin, from Hopf's Lemma (cf. for instance \cite[p. 330]{MR2597943}), it follows that
\[
\frac{\partial \phi}{\partial x_{m}} (0) \neq 0,
\]
which is a contradiction, since $\grad \phi (0) = 0$. Therefore, $\varphi$ is positive on $\partial M$. \qed

\end{proof}

\begin{proposition}\label{eigenfunction}
Let $M$ be a (connected) Riemannian manifold and $S = \Delta +V$ a Schr\"{o}dinger operator on $M$. If $\varphi \in \mathcal{D}(S^{N}) \smallsetminus\{0\}$ is a $\lambda_{0}^{N}(S)$-eigenfunction of $S^{N}$, then $\varphi$ is smooth and nowhere vanishing. Moreover, if $M$ has non-empty boundary, then $\nu(\varphi) = 0$ on $\partial M$.
\end{proposition}

\begin{proof}
Since $\varphi \in \mathcal{D}(S^{N})$, there exists $(f_{n})_{n \in \mathbb{N}} \subset C^{\infty}_{c}(M)$, such that $f_{n} \rightarrow \varphi$ in $H_{V}(M)$. Clearly, $| f_{n} |$ is Lipschitz and compactly supported. From Proposition \ref{Approximation}, it follows that $|f_{n}| \in H_{V}(M)$.
From Rademacher's Theorem, $|f_{n}|$ is almost everywhere differentiable. Therefore, we have $\|\grad |f_{n}|\| = \|\grad f_{n}\|$ almost everywhere, and in particular, $\mathcal{R}_{S}(|f_{n}|) = \mathcal{R}_{S}(f_{n})$. Since $(| f_{n} |)_{n \in \mathbb{N}}$ is bounded in $H_{V}(M)$, it has a weakly convergent subsequence in $H_{V}(M)$. Since $| f_{n} | \rightarrow | \varphi |$ in $L^{2}(M)$, it follows that $| \varphi | \in H_{V}(M)$, and after passing to a subsequence, we have that $|f_{n}| \rightharpoonup | \varphi |$ in $H_{V}(M)$. Hence, $\mathcal{R}_{S}(| \varphi |) = \lambda_{0}^{N}(S)$.

In particular, for any $f \in C^{\infty}_{c}(M)$, the function $t \mapsto \mathcal{R}_{S}(|\varphi| + tf)$, with $|t| < \varepsilon$, is differentiable and attains minimum for $t=0$. This yields that
\begin{equation}\label{eigenfunction weakly}
\int_{M} (\langle \grad |\varphi| , \grad f  \rangle + V |\varphi| f) = \lambda_{0}^{N}(S) \int_{M} |\varphi| f,
\end{equation}
for any $f \in C^{\infty}_{c}(M)$. From Elliptic Regularity Theory, it follows that $|\varphi| \in C^{\infty}(M^{\circ})$ and $S|\varphi| = \lambda_{0}^{N}(S) |\varphi|$ in $M^{\circ}$. From Lemma \ref{Eigenfunction auxiliary}, $| \varphi |$ is nowhere vanishing in the interior of $M$, and so is $\varphi$. If $M$ has empty boundary, this completes the proof.
If $M$ has non-empty boundary, then without loss of generality, we may assume that $\varphi$ is positive in the interior of $M$. Since $\varphi \in \mathcal{D}(S^{N})$ and $S^{N}\varphi = \lambda_{0}^{N}(S)\varphi$, from Elliptic Regularity Theory, it follows that $\varphi \in C^{\infty}(M)$. Moreover, from (\ref{eigenfunction weakly}) we have that
\[
\int_{\partial M} \nu(\varphi) f = \int_{M} f S \varphi - \int_{M} (\langle \grad \varphi , \grad f \rangle + V \varphi f) = 0,
\]
for any $f \in C^{\infty}_{c}(M)$. Therefore, $\nu(\varphi) = 0$ on $\partial M$, and from Lemma \ref{Eigenfunction auxiliary}, it follows that $\varphi$ is positive on $\partial M$. \qed
\end{proof}

\begin{proposition}\label{Decomposition}
Let $S$ be a Schr\"{o}dinger operator on a (connected) Riemannian manifold $M$, with $\lambda_{0}^{N}(S) \notin \sigma_{\ess}^{N}(S)$. Then for any compact $K \subset M$ of positive measure, we have
\[
\inf_{f} \mathcal{R}_{S}(f) > \lambda_{0}^{N}(S),
\]
where the infimum is taken over all non-zero $f \in \Lip_{c}(M)$, with $\supp f \cap K = \emptyset$.
\end{proposition}

\begin{proof}
Assume to the contrary that there exists a compact subset $K$ of $M$ of positive measure, such that for any $\varepsilon > 0$, there exists a non-zero $f \in \Lip_{c}(M)$, with $\mathcal{R}_{S}(f) < \lambda_{0}^{N}(S) + \varepsilon$ and $\supp f \cap K = \emptyset$. Evidently, there exists $(f_{n})_{n \in \mathbb{N}} \subset \Lip_{c}(M)$, with $\| f_{n} \|_{L^{2}(M)} = 1$, $\supp f_{n} \cap K =\emptyset$ and $\mathcal{R}_{S}(f_{n}) \rightarrow \lambda_{0}^{N}(S)$. From Proposition \ref{minimizing sequence}, after passing to a subsequence, we have that $f_{n} \rightarrow \varphi$ in $L^{2}(M)$, for some $\lambda_{0}^{N}(S)$-eigenfunction $\varphi$ of $S^{N}$. Since $\| \varphi \|_{L^{2}(M)} = 1$, from Proposition \ref{eigenfunction}, it follows that $\varphi$ is nowhere vanishing in $M$. This is a contradiction, since
\[
\| \varphi - f_{n} \|_{L^{2}(M)}^{2} \geq \int_{K} \varphi^{2} > 0,
\]
while $f_{n} \rightarrow \varphi$ in $L^{2}(M)$. This proves the asserted claim. \qed
\end{proof}\medskip


We end this section with the notion of renormalized Schr\"{o}dinger operators. This notion was introduced for the Laplacian in \cite{Brooks2} and for Schr\"{o}dinger operators on manifolds without boundary in \cite{Mine}.

Let $M$ be a possibly non-connected Riemannian manifold and $S = \Delta + V$ a Schr\"{o}dinger operator on $M$. Let $\varphi \in C^{\infty}(M)$ be a positive function, satisfying $S \varphi  = \lambda \varphi$, for some $\lambda \in \mathbb{R}$. If $M$ has non-empty boundary, assume that $\nu(\varphi) = 0$ on $\partial M$. Consider the space $L^{2}_{\varphi}(M) := \{ [v] : \varphi v \in L^{2}(M) \}$, where two measurable functions are equivalent if they are almost everywhere equal, endowed with the inner product $\langle v_{1} ,v_{2} \rangle_{L^{2}_{\varphi}(M)} := \int_{M} v_{1} v_{2} \varphi^{2}$. Then the map $\mu_{\varphi} \colon L^{2}_{\varphi}(M) \to L^{2}(M)$, defined by $\mu_{\varphi}v := \varphi v$, is an isometric isomorphism. In particular, $L^{2}_{\varphi}(M)$ is a separable Hilbert space. The \textit{renormalization} $S_{\varphi}$ of $S$ with respect to $\varphi$ is defined by
\[
S_{\varphi}v := \mu_{\varphi}^{-1} (S^{N} - \lambda)(\mu_{\varphi}v), \text{ for all } v \in \mathcal{D}(S_{\varphi}) := \mu_{\varphi}^{-1}(\mathcal{D}(S^{N})).
\]
It is clear that $S_{\varphi} \colon \mathcal{D}(S_{\varphi}) \colon L^{2}_{\varphi}(M) \to L^{2}_{\varphi}(M)$ is self-adjoint and $\sigma(S_{\varphi}) = \sigma^{N}(S) - \lambda$. For a non-zero $f \in \Lip_{c}(M)$, the \textit{Rayleigh quotient} of $f$ with respect to $S_{\varphi}$ is defined as
\begin{equation*}\label{Rayleigh quotient renormalized}
\mathcal{R}_{S_{\varphi}}(f) := \frac{\int_{M} \| \grad f \|^{2} \varphi^{2}}{\int_{M} f^{2}\varphi^{2}}.
\end{equation*}

\begin{proposition}[{{\cite[Proposition 7.1]{Mine},\cite[Subsection 2.1]{BMP2}}}]\label{renormalization without boundary}
In the above situation, if $M$ has empty boundary, then the bottom of the spectrum of $S_{\varphi}$ is given by
\[
\lambda_{0}(S) - \lambda = \lambda_{0}(S_{\varphi}) = \inf_{f \in C^{\infty}_{c}(M) \smallsetminus \{0\}} \mathcal{R}_{S_{\varphi}}(f) = \inf_{f \in \Lip_{c}(M) \smallsetminus \{0\}} \mathcal{R}_{S_{\varphi}}(f).
\]
\end{proposition}

\begin{proposition}\label{renormalization Neumann}
In the above situation, if $M$ has non-empty boundary, then the bottom of the spectrum of $S_{\varphi}$ is given by
\[
\lambda_{0}^{N}(S) - \lambda = \lambda_{0}(S_{\varphi}) = \inf_{f} \mathcal{R}_{S_{\varphi}}(f),
\]
where the infimum is taken over all non-zero $f \in C^{\infty}_{c}(M)$, with $\nu(f) = 0$ on $\partial M$.
\end{proposition}

\begin{proof}
Let $f \in C^{\infty}_{c}(M)\smallsetminus\{0\}$, with $\nu(f) = 0$ on $\partial M$. Since $\varphi$ is smooth and $\nu(\varphi) = 0$ on $\partial M$, it follows that $f \in \mathcal{D}(S_{\varphi})$. It is easy to see that
\[
S_{\varphi}f = \Delta f - \frac{2}{\varphi}\langle \grad \varphi , \grad  f \rangle.
\]
Hence, we have
\begin{eqnarray}
\langle S_{\varphi} f , f \rangle_{L^{2}_{\varphi}(M)} &=& \int_{M} (\varphi^{2}f \Delta f - 2 f \varphi \langle \grad f, \grad \varphi \rangle) \nonumber \\
&=& \int_{M} (\langle \grad (\varphi^{2}f), \grad f \rangle - 2 f \varphi \langle \grad f, \grad \varphi \rangle) + \int_{\partial M} \varphi^{2} f \nu(f) \nonumber \\
&=& \int_{M} \| \grad f \|^{2} \varphi ^{2},\nonumber
\end{eqnarray}
where we used that $\nu(f) = 0$ on $\partial M$. In particular, we have that
\[
\mathcal{R}_{S_{\varphi}}(f) = \frac{\langle S_{\varphi} f ,f \rangle_{L^{2}_{\varphi}(M)}}{\| f \|_{L^{2}_{\varphi}(M)}^{2}}.
\]
From Proposition \ref{Rayleigh}, it follows that $\mathcal{R}_{S_{\varphi}}(f) \geq \lambda_{0}(S_{\varphi})$.
From Proposition \ref{Bottom of Friedrichs}, there exists $(g_{n})_{n \in \mathbb{N}} \subset C^{\infty}_{c}(M) \smallsetminus \{0\}$, with $\nu(g_{n})=0$ on $\partial M$ and $\mathcal{R}_{S}(g_{n}) \rightarrow \lambda_{0}^{N}(S)$. Consider $f_{n} := \mu_{\varphi}^{-1}g_{n}$. Then $f_{n} \in C^{\infty}_{c}(M)$, $\nu(f_{n}) = 0$ on $\partial M$, and $\mathcal{R}_{S_{\varphi}} (f_{n})\rightarrow \lambda_{0}(S_{\varphi})$. This proves the asserted equality. \qed

\end{proof}

\section{Coverings of compact manifolds}\label{Section Compact}

Throughout this section, for simplicity of notation, we denote by $\mathcal{R}(f)$ the Rayleigh quotient of a Lipschitz function $f$ with respect to the Laplacian. 
The aim of this section is to prove Theorem \ref{Theorem Neumann}. Since a part of it follows from Theorem \ref{Old result for Neumann}, it remains to prove the converse implication. For reasons that will become clear in the sequel, we need to establish it also for non-connected covering spaces.


\begin{theorem}\label{Neumann compact}
Let $p \colon M_{2} \to M_{1}$ be a Riemannian covering, with $M_{1}$ compact with non-empty boundary, and $M_{2}$ possibly non-connected. If $\lambda_{0}^{N}(M_{2}) = 0$, then $p$ is amenable.
\end{theorem}

Let $\nu_{i}$ be the inward pointing normal to $\partial M_{i}$, $i=1,2$. Then there exists $\delta > 0$, such that the map $\Phi \colon \partial M_{1} \times [0,2 \delta) \to M_{1}$, defined by $\Phi(x,t) := \exp_{x}(t\nu_{1})$, is a diffeomorphism onto its image. By definition, any point of $M_{1}$ has an evenly covered neighborhood with respect to the restriction of $p$ on any connected component of $M_{2}$. Therefore, we may assume that $\delta$ is sufficiently small, so that for any $x \in \partial M_{1}$ and $y_{1},y_{2} \in p^{-1}(x)$, with $y_{1} \neq y_{2}$, we have $d(y_{1},y_{2}) \geq 2\delta$. It is worth to point out that we consider the distance of points of different connected components of $M_{2}$ to be infinite.

\begin{lemma}\label{Psi injective}
The map $\Psi \colon \partial M_{2} \times [0,\delta) \to M_{2}$, defined by $\Psi(y,t) := \exp_{y}(t\nu_{2})$, is a diffeomorphism onto its image.
\end{lemma}

\begin{proof}
Since $(p \circ \Psi) (y,t) = \Phi(p(y),t)$, for any $y \in \partial M_{2}$ and $t \in [0,\delta)$, it is clear that $\Psi$ is a local diffeomorphism. So, it suffices to prove that it is injective. Consider $y_{1},y_{2} \in \partial M_{2}$ and $t_{1},t_{2} \in [0,\delta)$, such that $\Psi(y_{1},t_{1}) = \Psi(y_{2},t_{2}) =:z$. Then $d(y_{i},z) < \delta$, $i=1,2$, which yields that $d(y_{1},y_{2}) < 2 \delta$. Moreover, it follows that $\Phi(p(y_{1}) , t_{1}) = \Phi(p(y_{2}),t_{2})$. Since $\Phi$ is a diffeomorphism onto its image, this yields that $t_{1} = t_{2}$, $p(y_{1}) = p(y_{2})$, and in particular, $y_{1} = y_{2}$. \qed
\end{proof}

\begin{lemma}\label{metric}
There exists a Riemannian metric $\mathrm{g}^{\prime}$ on $M_{1}$, such that $\Phi$ restricted to $\partial M_{1} \times [0,\delta)$ is an isometry onto its image.
\end{lemma}

\begin{proof}
Let $\mathrm{g}_{c}$ be the push-forward of the product metric of $\partial M_{1} \times [0,2\delta)$ via $\Phi$. Denote by $\mathrm{g}$ the original Riemannian metric of $M_{1}$. Consider a smooth $\tau \colon [0,2\delta) \to [0,1]$, with $\tau(t) = 1$ for $t \leq \delta$, and $\tau (t) = 0$ for $t \geq 3\delta/2$. Consider the function $\tau^{\prime}\in C^{\infty}(M_{1})$, defined by $\tau^{\prime}(\Phi(x,t)) = \tau(t)$ in $\Phi(\partial M_{1} \times [0,2\delta))$, and $\tau^{\prime} = 0$ otherwise. Evidently, the Riemannian metric
\[
\mathrm{g}^{\prime} := \tau^{\prime} \mathrm{g}_{c} + (1-\tau^{\prime}) \mathrm{g}.
\]
on $M_{1}$ satisfies the desired property. \qed
\end{proof}\medskip

Consider $M_{1}$ and $M_{2}$ endowed with $\mathrm{g}^{\prime}$ and its lift, respectively. Since $(p \circ \Psi)(y,t) = \Phi (p(y),t)$, for any $y \in \partial M_{2}$ and $t\in[0,\delta)$, it follows that $\Psi$ restricted on $\partial M_{2} \times [0,\delta)$ is a local isometry, with respect to the lift of $\mathrm{g}^{\prime}$. From Lemma \ref{Psi injective}, this map is also injective, which yields that it is an isometry onto its image.
Denote by $U_{t}$ the open set $\Psi(\partial M_{2} \times [0,t))$, and by $C_{t}$ the closed set $\Psi(\partial M_{2} \times \{t\})$.

Evidently, there exist $c_{1}, c_{2} > 0$, such that for any $f \in C^{\infty}(M_{2})$, the norms of the gradients of $f$ with respect to the lifts of $\mathrm{g}$ and $\mathrm{g}^{\prime}$, are related by
\[
c_{1}\| {\grad} _{\mathrm{g}} f \|_{\mathrm{g}} \leq \| {\grad}_{\mathrm{g}^{\prime}} f \|_{\mathrm{g}^{\prime}} \leq c_{2} \| {\grad}_{\mathrm{g}} f \|_{\mathrm{g}}.
\]
Moreover, there exists a positive, smooth $\mathcal{V} \colon M_{1} \to \mathbb{R}$, such that the volume elements induced by the lifts of $\mathrm{g}$ and $\mathrm{g}^{\prime}$, satisfy
\[
\frac{d {\Vol} _{\mathrm{g}^{\prime}}}{d {\Vol}_{\mathrm{g}}} = \mathcal{V} \circ p.
\]
Therefore, for any non-zero $f \in C^{\infty}_{c}(M_{2})$, the Rayleigh quotients of $f$ with respect to the Laplacians induced by the lifts of $\mathrm{g}$ and $\mathrm{g}^{\prime}$, satisfy
\[
\mathcal{R}_{\mathrm{g}^{\prime}}(f) = \frac{\int_{M_{2}} \| \grad_{\mathrm{g}^{\prime}} f \|^{2}_{\mathrm{g}^{\prime}} d \Vol_{\mathrm{g}^{\prime}} }{\int_{M_{2}} f^{2} d \Vol_{\mathrm{g}^{\prime}}} \leq c_{2}^{2} \frac{\max \mathcal{V}}{\min \mathcal{V}} \mathcal{R}_{\mathrm{g}}(f).
\]
From Proposition \ref{Bottom of Neumann}, since $\lambda_{0}^{N}(M_{2}) = 0$ with respect to the lift of $\mathrm{g}$, it follows that $\lambda_{0}^{N}(M_{2}) = 0$ with respect to the lift of $\mathrm{g}^{\prime}$. From now on, we will be working with $\mathrm{g}^{\prime}$ and its lift. It is worth to point out that the maps $\Phi$ and $\Psi$ are defined in terms of the exponentials with respect to the original Riemannian metrics.

\begin{lemma}\label{approximation}
For any $\varepsilon > 0$, there exists $f \in \Lip_{c}(M_{2})$, smooth on $M_{2} \smallsetminus C_{t_{0}}$, for one $t_{0} \in (0,\delta)$, with $f|_{\partial M_{2}}$ non-zero, such that $\mathcal{R}(f) \leq \varepsilon$ and
\[
\frac{\int_{\partial M_{2}} \| \grad f \|^{2}}{\int_{\partial M_{2}} f^{2}} \leq \varepsilon.
\]
\end{lemma}

\begin{proof}
Since $\lambda_{0}^{N}(M_{2}) = 0$, from Proposition \ref{Bottom of Neumann}, there exists $(f_{n})_{n \in \mathbb{N}} \subset C^{\infty}_{c}(M_{2})$, with $\| f_{n} \|_{L^{2}(M_{2})} = 1$, such that $\mathcal{R}(f_{n}) \rightarrow 0$. Assume that there exists $\varepsilon > 0$, such that for any $n \in \mathbb{N}$ and $t \in [0,\delta)$, we have
\begin{equation}\label{contradiction}
\int_{C_{t}} \| \grad f_{n} \|^{2} > \varepsilon \int_{C_{t}} f^{2}_{n}.
\end{equation}
Then
\[
\int_{U_{\delta}} \| \grad f_{n} \|^{2} > \varepsilon \int_{U_{\delta}} f_{n}^{2},
\]
which yields that $\int_{U_{\delta}} f_{n}^{2} \rightarrow 0$ and $\int_{M_{2} \smallsetminus U_{\delta}} f_{n}^{2} \rightarrow 1$.
Let $\chi \in C^{\infty}(M_{1})$, with $\chi(x) = 1$ for $d(x, \partial M_{1}) \geq \delta$, and $\chi(x) = 0$ for $d(x,\partial M_{1}) < \delta / 2$. Let $\tilde{\chi} \in C^{\infty}(M_{2})$ be the lift of $\chi$. Then $\tilde{\chi} = 0$ in $U_{\delta/2}$ and $\tilde{\chi} = 1$ outside $U_{\delta}$. For $g_{n} := \tilde{\chi} f_{n} \in C^{\infty}_{c}(M_{2})$, we have


\[
\| g_{n} \|_{L^{2}(M_{2})}^{2} = \int_{U_{\delta}} \tilde{\chi}^{2} f_{n}^{2} + \int_{M_{2} \smallsetminus U_{\delta}} f_{n}^{2} \rightarrow 1,
\]
and
\[
\int_{M_{2}}\| \grad g_{n} \|^{2} \leq 2 \int_{U_{\delta}}( \tilde{\chi}^{2} \| \grad f_{n} \|^{2} + f_{n}^{2} \| \grad \tilde{\chi}\| ^{2}) + \int_{M_{2} \smallsetminus U_{\delta}} \| \grad f_{n} \|^{2} \rightarrow 0.
\]
In particular, we have that $\mathcal{R}(g_{n}) \rightarrow 0$. Since $g_{n}$ is supported in the interior of $M_{2}$, for any $n \in \mathbb{N}$,
from Proposition \ref{Bottom of Neumann} and Remark \ref{Dirichlet}, it follows that $\lambda_{0}^{D}(M_{2}) = 0$. This is a contradiction, since from Proposition \ref{Inequality of Bottoms} and Remark \ref{Dirichlet}, we have $\lambda_{0}^{D}(M_{2}) \geq \lambda_{0}^{D}(M_{1}) > 0$.

Hence, (\ref{contradiction}) cannot hold, that is, for any $\varepsilon >0$, there exists $n \in \mathbb{N}$ and $t \in [0,\delta)$, such that
\begin{equation}\label{estimate}
\int_{C_{t}} \| \grad f_{n} \|^{2} \leq \varepsilon \int_{C_{t}} f_{n}^{2}.
\end{equation}
Let $0 < \varepsilon < \lambda_{0}^{D}(M_{2})$ and consider $f_{n} \in C^{\infty}_{c}(M_{2})$, with $\| f_{n} \|_{L^{2}(M_{2})} = 1$, $\mathcal{R}(f_{n}) < \varepsilon$, satisfying (\ref{estimate}) for some $t  \in [0,\delta)$. Let $t_{0}$ be the minimum of all $t \in [0,\delta)$, for which (\ref{estimate}) holds. If $t_{0} = 0$, then $f_{n}$ is the desired function.
Otherwise, define $f \in C_{c}(M_{2})$ by $f = f_{n}$ outside $U_{t_{0}}$, and $f(\Psi (x,t)) = f_{n}(\Psi(x,t_{0}))$ for $t \leq t_{0}$. It is clear that $f \in \Lip_{c}(M)$ and is smooth on $M_{2} \smallsetminus C_{t_{0}}$.

Since $\mathcal{R}(f_{n}) < \lambda_{0}^{D}(M_{2})$, from Proposition \ref{Bottom of Neumann} and Remark \ref{Dirichlet}, it follows that $f_{n}$ is not identically zero on $U_{t_{0}}$. Since $\mathcal{R}(f_{n}) < \varepsilon$, from the definition of $t_{0}$, it follows that $f_{n}$ is not identically zero on $M_{2} \smallsetminus U_{t_{0}}$. In particular, this yields that $f$ is non-zero.
Since (\ref{estimate}) holds for $t=t_{0}$, we have
\[
\int_{\partial M_{2}} \| \grad f \|^{2} = \int_{C_{t_{0}}} \| \grad (f_{n}|_{C_{t_{0}}}) \|^{2} \leq \varepsilon \int_{C_{t_{0}}} f_{n}^{2} = \varepsilon \int_{\partial M_{2}} f ^{2}.
\]
Furthermore, we have
\begin{eqnarray}\label{Rayleigh quotient estimate}
\mathcal{R}(f) &=& \frac{\int_{0}^{t_{0}} \int_{C_{t}} \| \grad f \|^{2}  + \int_{M_{2} \smallsetminus U_{t_{0}}} \| \grad f_{n} \|^{2}}{\int_{0}^{t_{0}} \int_{C_{t}} f ^{2}  + \int_{M_{2} \smallsetminus U_{t_{0}}} f_{n}^{2}} \nonumber \\
&\leq& \frac{ \varepsilon \int_{0}^{t_{0}} \int_{C_{t}}  f ^{2}  + \int_{M_{2} \smallsetminus U_{t_{0}}} \| \grad f_{n} \|^{2}}{\int_{0}^{t_{0}} \int_{C_{t}} f ^{2}  + \int_{M_{2} \smallsetminus U_{t_{0}}} f_{n}^{2}} \nonumber \\
&\leq& \max \left\{ \varepsilon , \frac{\int_{M_{2} \smallsetminus U_{t_{0}}} \| \grad f_{n} \|^{2}}{\int_{M_{2} \smallsetminus U_{t_{0}}} f_{n}^{2}}  \right\}.
\end{eqnarray}
It is clear that
\begin{eqnarray}
\varepsilon &>& \mathcal{R}(f_{n}) = \frac{\int_{0}^{t_{0}} \int_{C_{t}} \| \grad f_{n} \|^{2}  + \int_{M_{2} \smallsetminus U_{t_{0}}} \| \grad f_{n} \|^{2}}{\int_{0}^{t_{0}} \int_{C_{t}} f_{n} ^{2}  + \int_{M_{2} \smallsetminus U_{t_{0}}} f_{n}^{2}} \nonumber \\
&\geq& \min \left\{ \frac{\int_{0}^{t_{0}} \int_{C_{t}} \| \grad f_{n} \|^{2}}{\int_{0}^{t_{0}} \int_{C_{t}} f_{n} ^{2}} ,\frac{\int_{M_{2} \smallsetminus U_{t_{0}}} \| \grad f_{n} \|^{2}}{\int_{M_{2} \smallsetminus U_{t_{0}}} f_{n}^{2}} \right\}. \nonumber
\end{eqnarray}
From the definition of $t_{0}$, the first term is greater than $\varepsilon$, which yields that the second term is smaller than $\varepsilon$. From (\ref{Rayleigh quotient estimate}), it follows that $\mathcal{R}(f) \leq \varepsilon$. Since $\varepsilon < \lambda_{0}^{D}(M_{2})$, from Remark \ref{Dirichlet} and Proposition \ref{Bottom of Neumann}, it is clear that $f$ cannot vanish identically on $\partial M_{2}$.\qed
\end{proof}\medskip

Glue the cylinder $\partial M_{1} \times [0,+\infty)$, with the product metric, along $\partial M_{1}$, so that $\partial / \partial t$ is the outward pointing normal to $\partial M_{1}$. Denote by $N_{1}$ the obtained Riemannian manifold.
The covering $p \colon M_{2} \to M_{1}$ can be extended to a Riemannian covering $p \colon N_{2} \to N_{1}$, where $N_{2}$ is the Riemannian manifold obtained by gluing $\partial M_{2} \times [0,+\infty)$ along $\partial M_{2}$ in the analogous way. Evidently, $p \colon M_{2} \to M_{1}$ is amenable if and only if $p \colon N_{2} \to N_{1}$ is amenable. Points in $N_{i} \smallsetminus M_{i}^{\circ}$ will be written in the form $(x,t)$, with $x \in \partial M_{i}$ and $t \geq 0$, $i=1,2$.

Consider a positive smooth $\phi \colon [0,+\infty) \to \mathbb{R}$, with $\phi(t) = 1$ for $t \leq 1/2$, and $\phi(t) = e^{-t}$ for $t \geq 1$. Let $\varphi \in C^{\infty}(N_{1})$ be the square-integrable function defined by $\varphi = 1$ in $M_{1}$, and $\varphi(x,t) = \phi(t)$ in $N_{1} \smallsetminus M_{1}$. Consider the function $V \in C^{\infty}(N_{1})$, defined by $V=0$ in $M_{1}$, and $V(x,t) = \phi^{\prime \prime}(t)/\phi(t)$ in $N_{1} \smallsetminus M_{1}$. It is worth to point out that outside the compact set $M_{1} \cup (\partial M_{1} \times [0,1])$, we have that $V=1$ and in particular, $V$ is bounded from below.
Consider the Schr\"{o}dinger operator $S_{1} = \Delta + V$ on $N_{1}$ and its lift $S_{2}$ on $N_{2}$. It is clear that $S_{1} \varphi = 0$.

\begin{remark}\label{construction remark}
Evidently, $N_{1}$ is complete, without boundary and with Ricci curvature bounded from below. Since $V=1$ outside the compact set $M_{1} \cup (\partial M_{1} \times [0,1])$, from Propositions \ref{Bottom of essential spectrum} and \ref{Bottom of Neumann}, it follows that $\lambda_{0}^{\ess}(S_{1}) \geq 1$. Moreover, it is clear that $V$ and $ \grad V $ are bounded.
\end{remark}

\begin{lemma}\label{bottom zero}
The function $\varphi$ belongs to the domain of the Friedrichs extension of $S_{1}$ and in particular, $\lambda_{0}(S_{1}) = 0$.
\end{lemma}

\begin{proof}
For $T>0$, consider the compactly supported Lipschitz function $\chi_{T}$ defined by $\chi_{T} = 1$ in $M_{1}$, $\chi_{T}(x,t) = 1$ for $t \leq T$, $\chi_{T}(x,t) = T+1 - t$ for $T \leq t \leq T+1$, and $\chi_{T}(x,t) = 0$ for $t \geq T$. Then $\chi_{T} \varphi \in H_{0}^{1}(N_{1})$, for any $T>0$, and
\[
\| \varphi - \chi_{T} \varphi \|_{L^{2}(N_{1})}^{2} \leq \int_{\partial M_{1} \times [T,+\infty)} \varphi^{2}.
\]
Moreover, we have
\begin{eqnarray}
\int_{N_{1}} \| \grad (\varphi - \chi_{T} \varphi) \|^{2} &\leq& 2 \int_{ N_{1}} ((1-\chi_{T})^{2} \| \grad \varphi \|^{2} + \varphi^{2} \| \grad (1-\chi_{T}) \|^{2}) \nonumber \\
&\leq& 2 \int_{\partial M_{1} \times [T,+\infty)} \| \grad \varphi \|^{2} + 2 \int_{ \partial M_{1} \times [T,T+1] } \varphi^{2}. \nonumber
\end{eqnarray}
Evidently, $\varphi(x,t) = \| \grad \varphi (x,t) \| = e^{-t}$ for $t \geq 1$, which yields that $\chi_{T} \varphi \rightarrow \varphi$ in $H_{0}^{1}(N_{1})$, as $T \rightarrow +\infty$. Since $V$ is bounded, it follows that $\varphi \in H_{V}(N_{1})$. Since $S_{1} \varphi = 0$, it is clear that $ \varphi$ is an eigenfunction of the Friedrichs extension of $S_{1}$, which yields that $\lambda_{0}(S_{1}) \leq 0$. From Proposition \ref{Maximum of positive spectrum}, since $\varphi$ is positive, it follows that $\lambda_{0}(S_{1}) = 0$. \qed
\end{proof}\medskip


Denote by $\tilde{\varphi}$ the lift of $\varphi$ on $N_{2}$ and consider the renormalization $S_{\tilde{\varphi}}$ of $S_{2}$ with respect to $\tilde{\varphi}$. Let $g \in \Lip_{c}(M_{2})$, such that $g$ restricted on $\partial M_{2}$ is non-zero and smooth, and $h \colon [0,+\infty) \to  \mathbb{R}$ be a compactly supported, smooth function, with $h(t) = 1$ for $t \leq 1/2$. Extend $g$ in the glued ends $\partial M_{2} \times [0,+\infty)$ by
\begin{equation}\label{extension}
g(x,t) := g(x) h(t).
\end{equation}
It is clear that $g  \in \Lip_{c}(N_{2})$, and in the glued ends, we have
\[
\grad g (x,t) = g(x) h^{\prime}(t) \frac{\partial}{\partial t} + h(t)  \grad g(x).
\]
In particular, it follows that
\[
\| \grad g(x,t) \|^{2} = g^{2}(x) h^{\prime}(t)^{2} + h^{2}(t) \| \grad g (x) \|^{2},
\]
which yields that
\begin{eqnarray}\label{Rayleigh quotient of extension}
\frac{\int_{N_{2} \smallsetminus M_{2}} \| \grad g \|^{2} \tilde{\varphi}^{2}}{\int_{N_{2} \smallsetminus M_{2}} g^{2} \tilde{\varphi}^{2} } &=&  \frac{\int_{\partial M_{2}} \int_{0}^{+\infty} \| \grad g \|^{2}\tilde{\varphi}^{2}}{\int_{\partial M_{2}}\int_{0}^{+\infty} g^{2}\tilde{\varphi}^{2} } \nonumber \\
&=&  \frac{\int_{\partial M_{2}} \| \grad g \|^{2}}{\int_{\partial M_{2}} g^{2}} + \frac{\int_{0}^{+\infty} (h^{\prime})^{2} \phi^{2}}{\int_{0}^{\infty} h^{2} \phi^{2}},
\end{eqnarray}
where we used that in the glued ends $\partial M_{2} \times [0,+\infty)$, we have $\tilde{\varphi}(x,t) = \phi(t)$.
\begin{proposition}\label{equality of bottoms}
The renormalized operator $S_{\tilde{\varphi}}$ satisfies $\lambda_{0}(S_{\tilde{\varphi}})=0$, which yields that $\lambda_{0}(S_{2}) = \lambda_{0}(S_{1})$.
\end{proposition}

\begin{proof}
Let $\varepsilon > 0$. From Lemma \ref{approximation}, there exists $g \in \Lip_{c}(M_{2})$, smooth on $M_{2} \smallsetminus C_{t_{0}}$, for one $t_{0} \in (0,\delta)$, not vanishing identically on the boundary, such that
\[
\frac{\int_{M_{2}} \| \grad g \|^{2}}{\int_{M_{2}} g^{2}} < \frac{\varepsilon}{2}  \text{ and } \frac{\int_{\partial M_{2}} \| \grad g \|^{2}}{\int_{\partial M_{2}} g^{2}} < \frac{\varepsilon}{2}.
\]
Let $T > 1$ and consider a compactly supported, smooth $h \colon [0,+\infty) \to \mathbb{R}$, with $h(t) = 1$ for $t \leq T$, $h(t) = 0$ for $t \geq T + 1$, and $| h^{\prime} | \leq 2$.
Extend $g \in \Lip_{c}(M_{2})$ to the compactly supported $g \in \Lip_{c}(N_{2})$ as in (\ref{extension}). Evidently, we have
\[
\frac{\int_{0}^{+\infty} (h^{\prime})^{2} \phi^{2}}{\int_{0}^{+\infty} h^{2}\phi^{2}} \leq 4 \frac{\int_{T}^{T+1} e^{-2t}dt }{\int_{1}^{T} e^{-2t}dt} = 4 \frac{1 - e^{2}}{e^{2} - e^{2T}} < \frac{\varepsilon}{2},
\]
for some sufficiently large $T$. From (\ref{Rayleigh quotient of extension}), it follows that
\[
\frac{\int_{N_{2} \smallsetminus M_{2}} \| \grad g \|^{2} \tilde{\varphi}^{2}}{\int_{N_{2} \smallsetminus M_{2}} g^{2}\tilde{\varphi} ^{2}} < \varepsilon.
\]
Hence, we have
\begin{eqnarray}
\mathcal{R}_{S_{\tilde{\varphi}}}(g) &=&  \frac{\int_{M_{2}} \| \grad g \|^{2} +  \int_{N_{2} \smallsetminus M_{2}} \| \grad g \|^{2} \tilde{\varphi}^{2}}{\int_{M_{2}} g^{2} + \int_{N_{2} \smallsetminus M_{2}} g^{2}\tilde{\varphi}^{2}} \nonumber\\
&\leq& \max \left\{ \frac{\int_{M_{2}} \| \grad g \|^{2}}{\int_{M_{2}} g^{2}}, \frac{\int_{N_{2} \smallsetminus M_{2}} \| \grad g \|^{2} \tilde{\varphi}^{2}}{\int_{N_{2} \smallsetminus M_{2}} g^{2}\tilde{\varphi}^{2}}  \right\} < \varepsilon. \nonumber
\end{eqnarray}
Since $\varepsilon > 0$ is arbitrary, from Proposition \ref{renormalization without boundary}, it follows that $\lambda_{0}(S_{\tilde{\varphi}}) = 0$ and in particular, $\lambda_{0}(S_{2}) = \lambda_{0}(S_{1})$. \qed
\end{proof}\medskip

\noindent\emph{Proof of Theorem \ref{Neumann compact}:} Consider a Riemannian metric on $M_{1}$ as in Lemma \ref{metric} and its lift on $M_{2}$. Glue cylinders along the boundaries and extend the covering $p \colon M_{2} \to M_{1}$ to a Riemannian covering $p \colon N_{2} \to N_{1}$ as above. From Remark \ref{construction remark}, $N_{1}$ is complete, without boundary, and with Ricci curvature bounded from below. Consider the Schr\"{o}dinger operator $S_{1} = \Delta + V$ on $N_{1}$, as above, and its lift $S_{2}$ on $N_{2}$. From Remark \ref{construction remark}, we have that $V$ and $ \grad V $ are bounded.
From Lemma \ref{bottom zero} and Proposition \ref{equality of bottoms}, we obtain that $\lambda_{0}(S_{2}) = \lambda_{0}(S_{1}) = 0$, and Remark \ref{construction remark} yields that $\lambda_{0}^{\ess}(S_{1}) \geq 1$. From Theorem \ref{weak converse}, it follows that the covering $p \colon N_{2} \to N_{1}$ is amenable, and so is the covering $p \colon M_{2} \to M_{1}$. \qed \medskip



\noindent\emph{Proof of Theorem \ref{Theorem Neumann}:} Follows from Theorems \ref{Neumann compact} and \ref{Old result for Neumann}. \qed

\section{Arbitrary Riemannian coverings}\label{Section Arbitrary}

In this section, we prove Theorem \ref{Main result intro} and present some immediate consequences of it. As stated in the Introduction, we establish the following more general version of this theorem, involving manifolds with possibly non-empty boundary.

\begin{theorem}\label{main theorem}
Let $p \colon M_{2} \to M_{1}$ be a Riemannian covering. Let $S_{1}$ be a Schr\"{o}dinger operator on $M_{1}$, with $\lambda_{0}^{N}(S_{1}) \notin \sigma_{\ess}^{N}(S_{1})$, and $S_{2}$ its lift on $M_{2}$. Then $\lambda_{0}^{N}(S_{2}) = \lambda_{0}^{N}(S_{1})$ if and only if the covering is amenable.
\end{theorem}

The following lemma, which is a consequence of Theorem \ref{Neumann compact}, is essential for the proof of this theorem.

\begin{lemma}\label{key}
Let $p \colon M_{2} \to M_{1}$ be a non-amenable Riemannian covering. Let $S_{1}$ be a Schr\"{o}dinger operator on $M_{1}$, with $\lambda_{0}^{N}(S_{1})$ being an eigenvalue of $S_{1}^{N}$, and $S_{2}$ its lift on $M_{2}$.
If $\lambda_{0}^{N}(S_{2}) = \lambda_{0}^{N}(S_{1})$, then there exists a compact $K \subset M_{1}$ with non-empty interior, and $(f_{n})_{n \in \mathbb{N}} \subset C^{\infty}_{c}(M_{2})$, with $\| f_{n} \|_{L^{2}(M_{2})} = 1$, $\supp f_{n} \cap p^{-1}(K) = \emptyset$, for any $n \in \mathbb{N}$, and $\mathcal{R}_{S_{2}}(f_{n}) \rightarrow \lambda_{0}^{N}(S_{2})$.
\end{lemma}

\begin{proof}
If $M_{1}$ has non-empty boundary, then we denote by $\nu_{i}$ the inward pointing normal to $\partial M_{i}$, $i=1,2$.
From Proposition \ref{amenability domains}, since $p \colon M_{2} \to M_{1}$ is non-amenable, there exists a smoothly bounded, compact domain $K^{\prime}$, with non-empty interior, such that the covering $p \colon p^{-1}(K^{\prime}) \to K^{\prime}$ is non-amenable, where $p^{-1}(K^{\prime})$ may be non-connected. From Theorem \ref{Neumann compact}, it follows that $\lambda_{0}^{N}(p^{-1}(K^{\prime})) > 0$. 

Since $\lambda_{0}^{N}(S_{1})$ is an eigenvalue of $S_{1}^{N}$, from Proposition \ref{eigenfunction}, there exists a positive function $\varphi \in C^{\infty}(M_{1})$, with $S_{1} \varphi = \lambda_{0}^{N}(S_{1}) \varphi$ and $\nu_{1}(\varphi) = 0$ on $\partial M_{1}$ (if non-empty).
Consider the lift $\tilde{\varphi}$ of $\varphi$ on $M_{2}$ and the renormalization $S_{\tilde{\varphi}}$ of $S_{2}$ with respect to $\tilde{\varphi}$. Since $\lambda_{0}^{N}(S_{2}) = \lambda_{0}^{N}(S_{1})$, from Propositions \ref{renormalization without boundary} and \ref{renormalization Neumann}, it follows that
\[
0 = \lambda_{0}(S_{\tilde{\varphi}}) = \inf_{f} \frac{\int_{M_{2}} \| \grad f \|^{2} \tilde{\varphi}^{2}}{\int_{M_{2}} f^{2} \tilde{\varphi}^{2}},
\]
where the infimum is taken over all non-zero $f \in C^{\infty}_{c}(M_{2})$, with $\nu_{2}(f) = 0$ on $\partial M_{2}$ (if non-empty). In particular, there exists $(f_{n})_{n \in \mathbb{N}} \subset C^{\infty}_{c}(M_{2})$, with $\| f_{n} \|_{L^{2}_{\tilde{\varphi}}(M_{2})} = 1$, $\mathcal{R}_{S_{\tilde{\varphi}}}(f_{n}) \rightarrow 0$ and $\nu_{2}(f_{n}) = 0$ on $\partial M_{2}$ (if non-empty).

Since $\varphi$ is smooth and positive and $K^{\prime}$ is compact, there exist $c_{1},c_{2}>0$, such that $c_{1} \leq \varphi \leq c_{2}$ in $K^{\prime}$. From Proposition \ref{Bottom of Neumann}, it follows that
\[
\frac{\int_{p^{-1}(K^{\prime})} \| \grad f \|^{2} \tilde{\varphi}^{2}}{\int_{p^{-1}(K^{\prime})} f^{2} \tilde{\varphi}^{2}} \geq \frac{c_{1}^{2}}{c_{2}^{2}} \lambda_{0}^{N}(p^{-1}(K^{\prime})) > 0,
\]
for any $f \in C_{c}^{\infty}(p^{-1}(K^{\prime})) \smallsetminus \{0\}$. Since $\| f_{n} \|_{L^{2}_{\tilde{\varphi}}(M_{2})}=1$ and $\mathcal{R}_{S_{\tilde{\varphi}}}(f_{n}) \rightarrow 0$, it follows that
\[
\int_{p^{-1}(K^{\prime})} f_{n}^{2}  \tilde{\varphi}^{2} \rightarrow 0 \text{ and } \int_{M_{2} \smallsetminus p^{-1}(K^{\prime})} f_{n}^{2} \tilde{\varphi}^{2} \rightarrow 1.
\]
Let $K \subset M_{2}^{\circ}$ be a compact set, with non-empty interior, contained in the interior of $K^{\prime}$. Let $\chi \in C^{\infty}_{c}(M_{1})$, with $\chi = 1$ in a neighborhood of $K$, and $\supp \chi \subset K^{\prime} \cap M_{2}^{\circ}$. Consider the lift $\tilde{\chi}$ of $\chi$ on $M_{2}$, and let $g_{n} := (1 - \tilde{\chi}) f_{n} \in C^{\infty}_{c}(M_{2})$. It is clear that if $M_{1}$ has non-empty boundary, then $\nu_{2}(g_{n}) = 0$ on $\partial M_{2}$. Moreover, we have
$$\| g_{n} \|_{L^{2}_{\tilde{\varphi}}(M_{2})}^{2} = \int_{p^{-1}(K^{\prime})} (1 - \tilde{\chi})^{2} f_{n}^{2} \tilde{\varphi}^{2} + \int_{M_{2} \smallsetminus p^{-1}(K^{\prime})} f_{n}^{2} \tilde{\varphi}^{2} \rightarrow 1$$
and
\begin{eqnarray}
\int_{M_{2}} \| \grad g_{n} \|^{2} \tilde{\varphi}^{2} &\leq& 2 \int_{p^{-1}(K^{\prime})} (f_{n}^{2} \| \grad \tilde{\chi} \|^{2} +  (1-\tilde{\chi})^{2} \| \grad f_{n} \|^{2} )\tilde{\varphi}^{2} \nonumber \\
&+& \int_{M_{2} \smallsetminus p^{-1}(K^{\prime})} \|\grad f_{n}\|^{2} \tilde{\varphi}^{2} \rightarrow 0. \nonumber
\end{eqnarray} 
Therefore, $\mathcal{R}_{S_{\tilde{\varphi}}}(g_{n}) \rightarrow 0$ and $\supp g_{n} \cap p^{-1}(K) = \emptyset$. We may normalize $g_{n}$ in $L^{2}_{\tilde{\varphi}}(M)$, so that $\| g_{n} \|_{L^{2}_{\tilde{\varphi}} (M_{2})}=1$, for any $n \in \mathbb{N}$.

Consider $h_{n} := \tilde{\varphi} g_{n} \in C^{\infty}_{c}(M_{2})$. If $M_{2}$ has non-empty boundary, since $\nu_{2}(\tilde{\varphi}) = 0$ and $\nu_{2}(g_{n}) =0$, it follows that $\nu_{2}(h_{n}) = 0$ on $\partial M_{2}$. Evidently, $\| h_{n} \|_{L^{2}(M_{2})} = \| g_{n} \|_{L^{2}_{\tilde{\varphi}}(M_{2})}= 1 $. Moreover, from the definition of the renormalized Schr\"{o}dinger operator, it is clear that
\begin{eqnarray}
\mathcal{R}_{S_{2}}(h_{n}) &=& \langle S_{2}h_{n},h_{n} \rangle_{L^{2}(M_{2})} = \langle S_{\tilde{\varphi}}g_{n},g_{n} \rangle_{L_{\tilde{\varphi}}^{2}(M_{2})} + \lambda_{0}^{N}(S_{2}) \nonumber \\
&=& \mathcal{R}_{S_{\tilde{\varphi}}}(g_{n}) + \lambda_{0}^{N}(S_{2}) \rightarrow \lambda_{0}^{N}(S_{2}), \nonumber
\end{eqnarray}
which completes the proof. \qed
\end{proof}\medskip

\noindent\emph{Proof of Theorem \ref{main theorem}:}
From Theorem \ref{Old result for Neumann}, if the covering is infinite sheeted and amenable, then $\lambda_{0}^{N}(S_{1}) = \lambda_{0}^{N}(S_{2})$.
If the covering is finite sheeted, then for $f \in C^{\infty}_{c}(M_{1})$, we have that $f \circ p \in C^{\infty}_{c}(M_{2})$, and the equality of the bottoms follows from Proposition \ref{Bottom of Neumann} and Corollary \ref{Inequality of Bottoms}.
Hence, it remains to prove the converse implication.

Assume to the contrary that the covering is non-amenable. Since $\lambda_{0}^{N}(S_{2}) = \lambda_{0}^{N}(S_{1}) \notin \sigma_{\ess}^{N}(S_{1})$, from Lemma \ref{key} there exists a compact $K \subset M_{1}$ with non-empty interior, and $(f_{n})_{n \in \mathbb{N}} \subset C^{\infty}_{c}(M_{2})$, with $\| f_{n} \|_{L^{2}(M_{2})} = 1$, $\supp f_{n} \cap p^{-1}(K) = \emptyset$, for any $n \in \mathbb{N}$, and $\mathcal{R}_{S_{2}}(f_{n}) \rightarrow \lambda_{0}^{N}(S_{2})$.
For $n \in \mathbb{N}$, consider the pushdown $g_{n}$ of $f_{n}$, defined by
\[
g_{n}(z) := \big( \sum_{y \in p^{-1}(z)} f_{n}(y)^{2} \big)^{1/2},
\]
for any $z \in M_{1}$. Then $g_{n} \in \Lip_{c}(M_{1})$, $\| g_{n} \|_{L^{2}(M_{1})} = 1$ and $\mathcal{R}_{S_{1}}(g_{n}) \leq \mathcal{R}_{S_{2}}(f_{n})$, for any $n \in \mathbb{N}$ (cf. \cite[Section 4]{BMP1}). From Proposition \ref{Bottom of Neumann}, since $\lambda_{0}^{N}(S_{2})=\lambda_{0}^{N}(S_{1})$, it follows that $\mathcal{R}_{S_{1}}(g_{n}) \rightarrow \lambda_{0}^{N}(S_{1})$. From Proposition \ref{Decomposition}, since $\lambda_{0}^{N}(S_{1}) \notin \sigma_{\ess}^{N}(S_{1})$ and $\supp g_{n} \cap K = \emptyset$, this is a contradiction. Hence, the covering is amenable. \qed \medskip



\noindent{\emph{Proof of Theorem \ref{Main result intro}:}} Follows from Theorem \ref{main theorem}, since the manifolds involved may have empty boundary. \qed

\begin{remark}\label{main result Dirichlet}
In Theorem \ref{Main result intro}, the manifolds do not have to be complete. Therefore, from Remark \ref{Dirichlet}, we obtain the corresponding result for the Dirichlet spectrum of Schr\"{o}dinger operators on manifolds with boundary.
\end{remark}

\begin{corollary}
Let $p \colon M_{2} \to M_{1}$ be a Riemannian covering, with $M_{1}$ compact. Then the covering is amenable if and only if it preserves the bottom of the Dirichlet/Neumann spectrum of some/any Schr\"{o}dinger operator.
\end{corollary}

\begin{proof}
Follows from Theorem \ref{main theorem} and Remark \ref{main result Dirichlet}, since the Dirichlet and the Neumann spectrum of a Schr\"{o}dinger operator on a compact manifold is discrete. \qed
\end{proof}\medskip


The next example shows that the assumption $\lambda_{0}(S_{1}) \notin \sigma_{\ess}(S_{1})$ in Theorem \ref{Main result intro} cannot be replaced with $\lambda_{0}(S_{1})$ being an eigenvalue of the Friedrichs extension of $S_{1}$.

\begin{example}\label{Example}
Let $M_{1}$ be a two dimensional torus with a cusp attached, endowed with a Riemannian metric, such that $M_{1}$ is complete and outside a compact set, the cusp is the surface of revolution generated by $1/t^{2}$, with $t \geq 1$. Since $M_{1}$ has finite volume, it follows that $\lambda_{0}(M_{1}) = 0$ and constant functions are $\lambda_{0}(M_{1})$-eigenfunctions of the Friedrichs extension of the Laplacian on $M_{1}$.
Let $x$ be a point of the torus and consider the non-negative quantity
\[
\mu := - \lim_{r \rightarrow +\infty} \frac{1}{r} \ln (\Vol(M_{1}) - \Vol(B(x,r))) \leq - \lim_{r \rightarrow +\infty} \frac{1}{r} \ln (2 \pi \int_{r+1}^{+\infty} \frac{1}{t^{2}} dt) = 0.
\]
From \cite[Theorem 1]{MR757481}, it follows that $\lambda_{0}^{\ess}(M_{1}) = 0$.
Consider the universal covering $p \colon M_{2} \to M_{1}$. Since $\pi_{1}(M_{1})$ is the free group with two generators, it follows that $p$ is non-amenable. Since the fundamental group of the cusp is amenable, from \cite[Corollary 1.6]{Mine}, it follows that $\lambda_{0}(M_{2}) = 0$.
\end{example}

It is clear that Theorem \ref{Main result intro} is more general than the results of \cite{BMP2,Brooks2,RT}. For sake of completeness, we present an example demonstrating this fact.
Let $p \colon M_{2} \to M_{1}$ be a Riemannian covering, with $M_{1}$ non-compact, complete, without boundary, and with $\sigma_{\ess}(M_{1}) = \emptyset$. Then, from Theorem \ref{Main result intro}, it follows that $\lambda_{0}(M_{2}) = \lambda_{0}(M_{1})$ if and only if the covering is amenable. It is worth to point out that since we do not require the covering to be normal, the results of \cite{Brooks2,RT} cannot be applied in this case. Moreover, from \cite[Theorem 3.1]{MR592568}, it follows that the Ricci curvature of $M_{1}$ is not bounded from below. Hence, also the result of \cite{BMP2} cannot be applied in this case.



\section{An application}\label{Section Application}

The aim of this section is to prove the following proposition, which was established for the Laplacian on manifolds without boundary in \cite{BMP2}.

\begin{proposition}\label{Application}
Let $p \colon M_{2} \to M_{1}$ be an infinite sheeted Riemannian covering. If $\lambda_{0}^{N}(S_{1}) = \lambda_{0}^{N}(S_{2})$, then $\lambda_{0}^{N}(S_{2}) \in \sigma_{\ess}^{N}(S_{2})$.
\end{proposition}

The main point of this proposition is that the covering is not required to be normal (or to have infinite deck transformations group), since in this case, according to \cite[Corollary 1.4]{Mine}, the spectrum of $S_{2}^{N}$ coincides with its essential spectrum.
It is worth to point out that the manifolds in this proposition may have empty boundary.
Moreover, since they may be non-complete, from Remark \ref{Dirichlet}, the analogous statement holds for the Dirichlet spectrum of Schr\"{o}dinger operators on manifolds with boundary.

\begin{proposition}\label{minimizing sequence weak conv}
Let $S= \Delta +V$ be a Schr\"{o}dinger operator on a Riemannian manifold $M$ and $(f_{n})_{n \in \mathbb{N}} \subset \Lip_{c}(M)$, with $\| f_{n} \|_{L^{2}(M)} = 1$ and $\mathcal{R}_{S}(f_{n}) \rightarrow \lambda_{0}^{N}(S)$. If $\lambda_{0}^{N}(S)$ is not an eigenvalue of $S^{N}$, then there exists a subsequence $(f_{n_{k}})_{k \in \mathbb{N}}$, such that $f_{n_{k}} \rightharpoonup 0$ in $L^{2}(M)$.
\end{proposition}

\begin{proof}
From Proposition \ref{Approximation}, there exists $(f_{n}^{\prime}) \in C^{\infty}_{c}(M) \cap \mathcal{D}(S^{N})$, with $\| f_{n}^{\prime} \|_{L^{2}(M)} = 1$ and $\| f_{n} - f_{n}^{\prime} \|_{H_{V}(M)} \leq 1/n$, for any $n \in \mathbb{N}$, where $H_{V}(M)$ is the space defined in Subsection \ref{Schroedinger preliminaries}. It is clear that $\mathcal{R}_{S}(f_{n}^{\prime}) \rightarrow \lambda_{0}^{N}(S)$ and it suffices to prove the statement for $(f_{n}^{\prime})_{n \in \mathbb{N}}$. From the Spectral Theorem (cf. \cite[Chapter 8]{Taylor2}), there exists a measure space $X$, such that $L^{2}(M)$ is isometrically isomorphic to $L^{2}(X)$, and under this identification, $S^{N}$ corresponds to a multiplication operator with a measurable function $f \colon X \to \mathbb{R}$; that is, an operator of the form $\mu_{f} \colon \mathcal{D}(\mu_{f}) \subset L^{2}(X) \to L^{2}(X)$, with $\mathcal{D}(\mu_{f}) := \{ g \in L^{2}(X) : fg \in L^{2}(X) \}$ and $\mu_{f}(g) = fg$, for any $g \in \mathcal{D}(\mu_{f})$. 
The spectrum of $S^{N}$ coincides with the essential range of $f$ and in particular, $f \geq \lambda_{0}^{N}(S)$ almost everywhere.

Let $(g_{n})_{n \in \mathbb{N}} \subset \mathcal{D}(\mu_{f})$ be the sequence corresponding to $(f_{n}^{\prime})_{n \in \mathbb{N}}$ under this identification. Since $\| g_{n} \|_{L^{2}(X)} = 1$, after passing to a subsequence, we have that $g_{n} \rightharpoonup g$ in $L^{2}(X)$, for some $g \in L^{2}(X)$. It is clear that
\[
\int_{X} (f - \lambda_{0}^{N}(S)) g_{n}^{2} = \langle \mu_{f} g_{n},g_{n} \rangle_{L^{2}(X)} - \lambda_{0}^{N}(S) =  \mathcal{R}_{S}(f^{\prime}_{n})- \lambda_{0}^{N}(S) \rightarrow 0.
\]
For $\varepsilon > 0$, consider the measurable set $A_{\varepsilon} := \{ f \geq \lambda_{0}^{N}(S) + \varepsilon \}$. Evidently, we have
\[
\int_{ A_{\varepsilon} } g_{n}^{2} \leq \frac{1}{\varepsilon} \int_{A_{\varepsilon}} (f - \lambda_{0}(S^{N})) g_{n}^{2} \rightarrow 0.
\]
Since $g_{n} \rightharpoonup g$ in $L^{2}(X)$, this yields that $g= 0$ almost everywhere in $A_{\varepsilon}$. In particular, $g = 0$ almost everywhere in $X \smallsetminus f^{-1}(\{\lambda_{0}^{N}(S)\})$, which yields that $\mu_{f} g = \lambda_{0}^{N}(S) g$. Since $\lambda_{0}^{N}(S)$ is not an eigenvalue of $S^{N}$, it follows that $g=0$. Therefore, $g_{n} \rightharpoonup 0$ in $L^{2}(X)$, which yields that $f_{n}^{\prime} \rightharpoonup 0$ in $L^{2}(M)$.\qed
\end{proof}

\begin{lemma}\label{existence of eigenfunction}
Let $p \colon M_{2} \to M_{1}$ be a Riemannian covering. If $\lambda_{0}^{N}(S_{2}) = \lambda_{0}^{N}(S_{1}) \notin \sigma_{\ess}^{N}(S_{2})$, then $\lambda_{0}^{N}(S_{1})$ is an eigenvalue of $S_{1}^{N}$.
\end{lemma}

\begin{proof}
Assume to the contrary that $\lambda_{0}^{N}(S_{1})$ is not an eigenvalue of $S^{N}_{1}$. From Proposition \ref{eigenfunction}, there exists a square-integrable, $\lambda_{0}^{N}(S_{2})$-eigenfunction $\varphi$ of $S_{2}^{N}$, which is smooth and positive in $M_{2}$. Without loss of generality, we may assume that $\| \varphi \|_{L^{2}(M_{2})} = 1$. Since $\varphi \in H_{V \circ p}(M_{2})$, there exists $(f_{n})_{n \in \mathbb{N}} \subset C^{\infty}_{c}(M_{2})$, with $\| f_{n} \|_{L^{2}(M_{2})} = 1$ and $f_{n} \rightarrow \varphi$ in $H_{V \circ  p}(M_{2})$, where $H_{V \circ p}(M_{2})$ is the space defined in Subsection \ref{Schroedinger preliminaries}. Evidently, we have that $\mathcal{R}_{S_{2}}(f_{n}) \rightarrow \lambda_{0}^{N}(S_{2})$.


Consider the pushdowns
\[
g_{n}(z) := \big( \sum_{y \in p^{-1}(z)} f_{n}(y)^{2} \big)^{1/2}.
\]
on $M_{1}$, with $n \in \mathbb{N}$. Then $g_{n} \in \Lip_{c}(M_{1})$, $\| g_{n} \|_{L^{2}(M_{1})} = 1$ and $\mathcal{R}_{S_{1}}(g_{n}) \leq \mathcal{R}_{S_{2}}(f_{n})$, for any $n \in \mathbb{N}$ (cf. \cite[Section 4]{BMP1}). From Proposition \ref{Bottom of Neumann}, since $\lambda_{0}^{N}(S_{1})=\lambda_{0}^{N}(S_{2})$, it follows that $\mathcal{R}_{S_{1}}(g_{n}) \rightarrow \lambda_{0}^{N}(S_{1})$. Since $\lambda_{0}^{N}(S_{1})$ is not an eigenvalue of $S^{N}_{1}$, from Proposition \ref{minimizing sequence weak conv}, after passing to a subsequence, we have that $g_{n} \rightharpoonup 0$ in $L^{2}(M_{1})$.

Consider a non-negative $\chi_{2} \in C^{\infty}_{c}(M_{2}) \smallsetminus \{0\}$, and its pushdown $\chi_{1} \in \Lip_{c}(M_{1})$ on $M_{1}$. Then
\begin{eqnarray}
\langle \chi_{2} , f_{n} \rangle_{L^{2}(M_{2})} &=& \int_{M_{1}} \sum_{y \in p^{-1}(z)} \chi_{2}(y) f_{n}(y) dz \nonumber \\ 
&\leq& \int_{M_{1}} (\sum_{y \in p^{-1}(z)} \chi_{2}(y)^{2})^{1/2} (\sum_{y \in p^{-1}(z)} f_{n}(y)^{2})^{1/2} dz \nonumber \\
&=& \langle \chi_{1} , g_{n} \rangle_{L^{2}(M_{1})}. \nonumber
\end{eqnarray}
This is a contradiction, since $\langle \chi_{1} , g_{n} \rangle_{L^{2}(M_{1})} \rightarrow 0$ and $\langle \chi_{2} , f_{n} \rangle_{L^{2}(M_{2})} \rightarrow \int_{M_{2}} \chi_{2} \varphi > 0$. Therefore, $\lambda_{0}^{N}(S_{1})$ is an eigenvalue of $S_{1}^{N}$. \qed
\end{proof}\medskip

\noindent\emph{Proof of Proposition \ref{Application}:}
If the covering is amenable, then the claim follows from Theorem \ref{Old result for Neumann}. Hence, it remains to prove the statement for $p$ non-amenable. Assume to the contrary that $\lambda_{0}^{N}(S_{2}) \notin \sigma_{\ess}^{N}(S_{2})$. From Lemma \ref{existence of eigenfunction}, it follows that $\lambda_{0}^{N}(S_{1})$ is an eigenvalue of $S_{1}^{N}$. Since $\lambda_{0}^{N}(S_{2}) = \lambda_{0}^{N}(S_{1})$, from Lemma \ref{key}, there exists a compact set $K \subset M_{1}$ with non-empty interior, and $(f_{n})_{n \in \mathbb{N}} \subset C^{\infty}_{c}(M_{2}) \smallsetminus\{0\}$, such that $\mathcal{R}_{S_{2}}(f_{n}) \rightarrow \lambda_{0}^{N}(S_{2})$ and $\supp f_{n} \cap p^{-1}(K) = \emptyset$, for any $n \in \mathbb{N}$. 
From Proposition \ref{Decomposition}, since $\lambda_{0}^{N}(S_{2}) \notin \sigma^{N}_{\ess}(S_{2})$ and $p^{-1}(K)$ contains compact sets of positive measure, this is a contradiction. \qed

\noindent Max Planck Institute for Mathematics \\
Vivatsgasse 7, 53111, Bonn \\
E-mail address: polymerp@mpim-bonn.mpg.de

\end{document}